\documentclass[a4paper, 12pt]{scrartcl}

\usepackage{cite}
\usepackage[english]{babel}
\usepackage[utf8]{inputenc}
\usepackage[T1]{fontenc}
\usepackage{lmodern}
\usepackage{amsmath, amsthm, amssymb, mathtools}
\usepackage[fixlanguage]{babelbib}
\selectbiblanguage{english}
\usepackage[pdftex]{graphicx,color}
\usepackage{float}
\usepackage{caption}
\usepackage{subcaption}
\usepackage[dvipsnames,svgnames]{xcolor} %Colors
\usepackage{tikz,pgf}
\usetikzlibrary{calc}
\usepackage{tkz-euclide}
\usetkzobj{all} 
\usepackage[colorlinks=true,linkcolor=RoyalBlue,citecolor=PineGreen]{hyperref}

\usepackage{breqn}

\usepackage{booktabs}  
\usepackage{multirow}

\usepackage[textwidth=2\marginparwidth]{todonotes}

\usepackage{algorithm,algorithmicx,algpseudocode}

\newcommand{\blue}{\text{blue}}
\newcommand{\red}{\text{red}}

\newcommand{\norm}[1]{\left\lVert#1\right\rVert}

\DeclareMathOperator{\Upairs}{U}
\newcommand{\upairs}[1]{\Upairs(#1)}

\DeclareMathOperator{\CDC}{CDC}
\newcommand{\cdc}[1]{\CDC(#1)}

\DeclareMathOperator{\NAC}{NAC}
\newcommand{\nac}[1]{\NAC_{#1}}
\newcommand{\nacC}[2]{\NAC_{#1}(#2)}

\newcommand{\RR}{\mathbb{R}}
\newcommand{\CC}{\mathbb{C}}
\newcommand{\QQ}{\mathbb{Q}}
\newcommand{\ci}{i}

\newcommand{\C}{\mathcal C}
\newcommand{\setreal}[2]{\mathcal R(#1,#2)}

\newcommand{\conj}[1]{\overline{#1}}

\newtheorem{thm}{Theorem}[section]
\newtheorem{lem}[thm]{Lemma}
\newtheorem{cor}[thm]{Corollary}
\newtheorem*{rem}{Remark}

\theoremstyle{definition}
\newtheorem{defn}[thm]{Definition}
\newtheorem{exmp}[thm]{Example}

\newcommand{\implstyle}[1]{\texttt{#1}}

\newcommand{\FlexRiLoG}{\implstyle{FlexRiLoG}}
\newcommand{\sage}{\implstyle{SageMath}} 

\tikzstyle{vertex}=[circle, draw, fill=black, inner sep=0pt, minimum size=4pt]
\tikzstyle{smallvertex}=[circle, draw, fill=black, inner sep=0pt, minimum size=2pt]
\tikzstyle{edge}=[line width=1.5pt,black!50!white]
\tikzstyle{gridl}=[black!50!white]

\tikzstyle{lnode}=[circle,white,draw=black!60!white,fill=black!60!white,inner sep=1pt, font=\scriptsize]
\tikzstyle{redge}=[edge,Red]
\tikzstyle{bedge}=[edge,NavyBlue]

\colorlet{col1}{LimeGreen}
\colorlet{col2}{Orchid}
\colorlet{col3}{Orange}
\colorlet{col4}{Cerulean}
\colorlet{col5}{SlateGrey}
\colorlet{col6}{LightGray}

\begin{document}

\title{Graphs with Flexible Labelings allowing Injective Realizations}

\author{%
Georg Grasegger\thanks{Johann Radon Institute for Computational and Applied Mathematics (RICAM), Austrian Academy of Sciences}
\and 
Jan Legersk\'y\thanks{Research Institute for Symbolic Computation (RISC), Johannes Kepler University Linz}
\and 
Josef Schicho\footnotemark[2]
}

\date{}

\maketitle

\footnotetext{
	\mbox{}\\[0.5ex]
	\begin{minipage}[t]{0.8\textwidth}
		This project has received funding from the European Union's Horizon~2020 research and innovation programme under the Marie Sk\l{}odowska-Curie grant agreement No~675789.
	\end{minipage}\hfill 
	\begin{minipage}[t]{0.12\textwidth}
		\vspace{-1em}
		\includegraphics[width=0.95\textwidth]{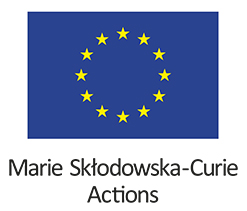}
	\end{minipage}\\
	The project was partially supported by the Austrian Science Fund (FWF): P31061, P31888, W1214-N15 (project DK9).
}

\begin{abstract}
	We consider realizations of a graph in the plane such that the distances
	between adjacent vertices satisfy the constraints given by an edge labeling.
	If there are infinitely many such realizations, counted modulo rigid motions,
	the labeling is called flexible.
	The existence of a flexible labeling, possibly non-generic,
	has been characterized combinatorially by the existence of a so called NAC-coloring.
	Nevertheless, the corresponding realizations are often non-injective.
	In this paper, we focus on flexible labelings with infinitely many injective realizations.
	We provide a necessary combinatorial condition on existence of 
	such a labeling based also on NAC-colorings of the graph.
	By introducing new tools for the construction of such labelings,
	we show that the necessary condition is also sufficient up to 8 vertices,
	but this is not true in general for more vertices.
\end{abstract}

\section{Introduction}
A widely studied question in Rigidity Theory is the number of realizations of a graph in $\RR^2$ 
such that the distances of adjacent vertices are equal to a given labeling of edges by positive real numbers.
Such a labeling is called flexible if the number of realizations, counted modulo rigid transformations, is infinite.
Otherwise, the labeling is called rigid.
We call a graph movable if there is a flexible labeling with infinitely many injective realizations, modulo rigid transformations.
In other words, we disallow realizations that identify two vertices; we do not care if edges intersect or even if edges
overlap in a line segment.
One can model such a movable graph as a planar linkage, where the vertices are rotational joints and the edges correspond to links of the length given by the labeling.

A result of Pollaczek-Geiringer~\cite{Geiringer1927}, rediscovered by Laman~\cite{Laman1970},
shows that a generic realization of a graph defines a rigid labeling 
if and only if the graph contains a Laman subgraph with the same set of vertices.
A graph $G=(V_G,E_G)$ is called Laman if $|E_G| = 2|V_G|-3$, and $|E_H|\leq 2|V_H|-3$
for all subgraphs $H$ of $G$ on at least two vertices.
Hence, if a graph is not spanned by a Laman graph, then a generic labeling is flexible, i.e., the graph is movable.

The study of movable overconstrained graphs has a long history.
Two ways of making the bipartite Laman graph $K_{3,3}$ movable
 were given by Dixon more than one hundred years ago~\cite{Dixon, WunderlichNineBar, Stachel}.
The first one works for any bipartite graph, placing the vertices of one part on the $x$-axis and of the other on the $y$-axis.
The second construction applies to $K_{4,4}$ and hence also to $K_{3,3}$.
Walter and Husty proved that these two give all flexible labelings of $K_{3,3}$ with injective realizations~\cite{WalterHusty}.
Other constructions are Burmester's focal point mechanisms~\cite{Burmester1893}, a graph with 9 vertices and 16 edges,
and two constructions by Wunderlich~\cite{Wunderlich1954,Wunderlich1981} for bipartite graphs based on geometric theorems.

The main question in this paper is the following: is a given graph movable?
In \cite{flexibleLabelings}, we already provide a combinatorial characterization of graphs
with a flexible labeling: there is a flexible labeling if and only if the graph has a so called NAC-coloring.
A NAC-coloring is a coloring of edges by two colors such that in every cycle,
either all edges have the same color or there are at least two edges of each color.
Many Laman graphs indeed have a NAC-coloring, but the corresponding realizations are in general not injective,
 i.e., in order to be flexible, some non-adjacent vertices coincide.
Here, we are more restrictive -- infinitely many realizations of a movable graph must be injective.

We give a necessary combinatorial condition on a graph being movable,
based on the concept of NAC-colorings.
The idea is that edges can be added to a graph if their endpoints are connected by a path that is monochromatic in every NAC-coloring,
without having effect on being movable. 
If the augmented graph does not have any NAC-coloring, it cannot be movable as it has no flexible labeling.
On the other hand, we provide constructions making some graphs movable.
They are based on NAC-colorings or combining movable subgraphs.
In combination with the necessary condition, we give a complete list of movable graphs up to 8 vertices.
Animations with the movable graphs can be found in~\cite{animations}.
The implementation of the concepts introduced in this paper is part of the \sage{} package \FlexRiLoG{} \cite{flexrilog}.

Figure~\ref{fig:movableVsFlexible} provides some examples illustrating the results:
the left graph has no NAC-coloring, hence, it has no flexible labeling.
The graph in the middle has a NAC-coloring, namely, it has a flexible labeling,
but it is not movable since it does not satisfy the necessary condition based on 
augmenting by edges whose endpoints are connected by a monochromatic path.
In other words, all motions require some vertices to coincide.
The third graph is movable using one of our constructions.

The structure of the paper is the following. In Section~\ref{sec:preliminaries},
we specify the system of equations describing the problem and recall the definition of NAC-coloring and some previous results.
A few technical lemmas about NAC-colorings are also proven.
In Section~\ref{sec:combinatorialCondition}, we prove the necessary combinatorial condition on being movable.
We list all graphs spanned by a Laman graph up to 8~vertices that satisfy it.
All these graphs are shown to be movable in Section~\ref{sec:constructions}.
Moreover, an example that the necessary condition is not sufficient is also presented.

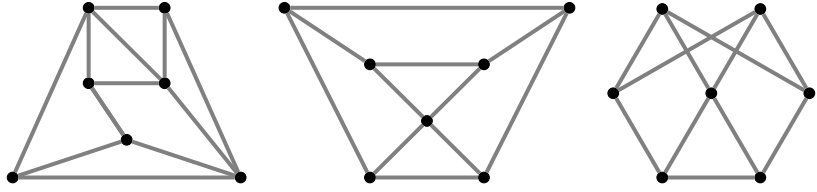
\begin{figure}[htb]
	\centering
	\begin{tabular}{ccc}
		\begin{tikzpicture}[scale=1]
			\node[vertex] (a) at (-0.5,-0.75) {};
			\node[vertex] (b) at (0.5,0.5) {};
			\node[vertex] (c) at (1.5,0.5) {};
			\node[vertex] (d) at (2.5,-0.75) {};
			\node[vertex] (e) at (0.5,1.5) {};
			\node[vertex] (f) at (1.5,1.5) {};
			\node[vertex] (g) at (1,-0.25) {};
			\draw[edge] (g)edge(b) (a)edge(e)   (g)edge(b)  (b)edge(c) (b)edge(e) (c)edge(d) ;
			\draw[edge] (c)edge(e) (c)edge(f) (d)edge(f) (e)edge(f) (a)edge(d) (g)edge(a) (g)edge(d);
		\end{tikzpicture}
		&
		\begin{tikzpicture}[scale=0.75]
		      \node[vertex] (1) at (-1,-1) {};
		      \node[vertex] (2) at (1,-1) {};
		      \node[vertex] (3) at (0,0) {};
		      \node[vertex] (4) at (-1,1) {};
		      \node[vertex] (5) at (1,1) {};
		      \node[vertex] (6) at (-2.5,2) {};
		      \node[vertex] (7) at (2.5,2) {};
		      \draw[edge] (1)edge(2) (1)edge(3) (2)edge(3)  (3)edge(5) (3)edge(4) (4)edge(5)  (5)edge(7) (2)edge(7);
		      \draw[edge] (1)edge(6) (4)edge(6) (6)edge(7);
		\end{tikzpicture}
		&
		\begin{tikzpicture}[scale=1.29]
		    \node[vertex] (5) at (0.500, 0.866) {};
		    \node[vertex] (4) at (-0.500, 0.866) {};
		    \node[vertex] (6) at (-1.00, 0.000) {};
		    \node[vertex] (3) at (1.00, 0.000) {};
		    \node[vertex] (2) at (0.500, -0.866) {};
		    \node[vertex] (0) at (-0.500, -0.866) {};
		    \node[vertex] (1) at (0.000, 0.000) {};
		    \draw[edge] (1)edge(2)  (1)edge(5)  (3)edge(5)  (2)edge(3)  (0)edge(2)  (4)edge(6)  (0)edge(1)    ;
		    \draw[edge] (0)edge(6)  (1)edge(4)  (3)edge(4)  (5)edge(6)   ;
		\end{tikzpicture}
	\end{tabular}
	\caption{The left graph has no flexible labeling, 
	the middle one has a flexible labeling, but it is not movable, and the right one is movable}
	\label{fig:movableVsFlexible}
\end{figure}

\section{Preliminaries}   \label{sec:preliminaries}
In the whole paper, all graphs are assumed to be connected and containing at least one edge.
We denote the set of vertices of a graph $G$ by $V_G$ and the set of edges by $E_G$.\linebreak[1]
In this section, we recall the definition of NAC-coloring and flexible labeling of a graph.
Next, we introduce the notion of proper flexible labeling and movable graph by the requirement of injective realizations.
We define an algebraic motion of a graph with a flexible labeling and assign a certain set of active NAC-colorings to this motion.
These active NAC-colorings come from the proof of the theorem characterizing the existence of a flexible labeling.
The active NAC-colorings are illustrated on the motion of a deltoid.
The section concludes with three lemmas,
which guarantee that the introduced notions are independent of certain choices of edges and interchanging colors.

\begin{defn}
	Let~$G$ be a graph and $\delta\colon  E_G\rightarrow \{\text{\blue{}, \red{}}\}$ be a coloring of edges.
	\begin{enumerate}
		\item A path, resp.\ cycle, in~$G$ is called \emph{monochromatic}, if all its edges have the same color.
		\item A cycle in~$G$ is an \emph{almost \red{} cycle}, resp.\ \emph{almost \blue{} cycle},
				if exactly one of its edges is \blue{}, resp.\ \red{}.
	\end{enumerate}
	A coloring~$\delta$ is called a \emph{NAC-coloring}, if it is surjective and there are no almost \blue{} cycles or almost \red{} cycles in~$G$.
	In other words every cycle is either monochromatic or contains at least 2 edges in each color.
	The set of all NAC-colorings of $G$ is denoted by $\nac{G}$.
\end{defn}
Now, the abbreviation NAC can be explained -- it stands for ``No Almost Cycle''.
Clearly, if we permute red and blue in a NAC-coloring of $G$, we obtain another NAC-coloring of~$G$.
\begin{defn}
	Let $G$ be a graph.
	If $\delta, \conj{\delta} \in \nac{G}$ are such that $\delta(e)=\blue{} \iff \conj{\delta}(e)=\red{}$ for all $e\in E_G$,
	then they are called \emph{conjugated}.
\end{defn}

The following definition describes the constraints on a realization in the plane given by a labeling of edges.
The realizations must be counted properly, i.e., modulo rigid motions, in order to say whether the labeling is flexible.
\begin{defn}
	Let~$G$ be a graph such that $|E_G|\geq 1$ and let $\lambda\colon E_G\rightarrow \RR_+$ be an edge labeling of~$G$.
	A map $\rho=(\rho_x,\rho_y)\colon V_G\rightarrow \RR^2$ is a \emph{realization of~$G$ compatible
	with~$\lambda$} iff $\norm{\rho(u)-\rho(v)}=\lambda(uv)$ for all edges~$uv\in E_G$.
	We say that two realizations~$\rho_1$ and~$\rho_2$ are congruent
	iff there exists a direct Euclidean isometry~$\sigma$ of~$\RR^2$ such that $\rho_1=\sigma \circ\rho_2$.

	The labeling $\lambda$ is called \emph{flexible}
	if the number of realizations of~$G$ compatible with~$\lambda$ up to congruence is infinite.
\end{defn}

We remark that if a labeling has a positive finite number of realizations, then it is called \emph{rigid}.

The constraints given by edge lengths~$\lambda_{uv}=\lambda(uv)$ can be modeled 
by the following system of equations for coordinates $(x_u,y_u)$ for $u\in V_G$.
In order to remove rigid motions, the position of an edge $\bar{u}\bar{v}$ is fixed:
\begin{align} \label{eq:mainSystemOfEquations}
	x_{\bar{u}}=0\,, \quad	y_{\bar{u}}&=0\,, \nonumber \\
	x_{\bar{v}}=\lambda_{\bar{u}\bar{v}}\,, \quad y_{\bar{v}}&=0\,, \\
	(x_u-x_v)^2+(y_u-y_v)^2&= \lambda_{uv}^2 \quad \text{ for all } uv \in E_G\setminus\{\bar{u}\bar{v}\}.	\nonumber
\end{align}
The labeling $\lambda$ is flexible if and only if there are infinitely many solutions of the system.

So far, the realizations have not been required to be injective. 
Namely, it could happen that two non-adjacent vertices were mapped to the same point in $\RR^2$.
Sections~\ref{sec:combinatorialCondition} and~\ref{sec:constructions} are focused on the graphs that have a 
labeling with infinitely many injective compatible realizations.
This corresponds to adding the inequalities	$(x_u-x_v)^2+(y_u-y_v)^2\neq 0$ for all $u,v \in V_G$
such that $u\neq v$ and $uv\notin E_G$.

\begin{defn}
	A flexible labeling $\lambda$ of a graph $G$ is called \emph{proper},
	if there exists infinitely many injective realizations $\rho$ of $G$ compatible with $\lambda$, modulo rigid transformations.
	We say that a graph is \emph{movable} if it has a proper flexible labeling.
\end{defn}
We remark that graphs that are not movable are called \emph{absolutely 2-rigid} in~\cite{Maehara1998}. 
Considering irreducible components of the solution set of the equation~\eqref{eq:mainSystemOfEquations}
allows us to use the notion of a function field, whose valuations give rise to a relation with NAC-colorings,
as we will see later.
\begin{defn}
  Let $\lambda$ be a flexible labeling of $G$.
  Let $\setreal{G}{\lambda}\subseteq (\RR^2)^{V_G}$ be the set of all realizations of $G$ compatible with $\lambda$.
  We say that $\C$ is an \emph{algebraic motion of $(G,\lambda)$ w.r.t.\ an edge $\bar{u}\bar{v}$},
  if it is an irreducible algebraic curve in $\setreal{G}{\lambda}$,
  such that $\rho(\bar{u})=(0,0)$ and $\rho(\bar{v})=(\lambda(\bar{u}\bar{v}),0)$ for all $\rho\in \C$.
  Since in many situations the role of  $\bar{u}\bar{v}$ does not matter, we also simply say that $\C$ is an algebraic motion of $(G,\lambda)$.
  We call $F(\C)$ the complex function field of $\C$.
\end{defn}
The fact that the choice of the fixed edge does not change the function field is proven at the end of this section.
The functions in the function field related to NAC-colorings are given by the following definition.

\begin{defn}
	Let $\lambda$ be a flexible labeling of a graph $G$. 
	Let $F(\C)$ be the complex function field of an algebraic motion $\C$ of $(G,\lambda)$. 
	For every $u,v\in V_G$ such that $uv\in E_G$, we define $W_{u,v}, Z_{u,v}\in F(\C)$ by
	\begin{align*}
		W_{u,v}&=(x_v-x_u) + \ci (y_v-y_u)\,, \\
		Z_{u,v}&=(x_v-x_u) - \ci (y_v-y_u)\,.
	\end{align*}	
	We use $W^{\bar{u}\bar{v}}_{u,v}$, resp.\ $Z^{\bar{u}\bar{v}}_{u,v}$, 
	if we want to specify that $\C$ is w.r.t.\ a fixed edge~$\bar{u}\bar{v}$.
\end{defn}
We remark that $W_{u,v}=-W_{v,u}$ and $Z_{u,v}=-Z_{v,u}$, i.e., they depend on the order of~$u,v$.
Using \eqref{eq:mainSystemOfEquations}, we have
\begin{align*}
	W_{\bar{u},\bar{v}}&=\lambda_{\bar{u}\bar{v}}\,,\quad Z_{\bar{u},\bar{v}}=\lambda_{\bar{u}\bar{v}}\,,\quad  \nonumber \\ 
	W_{u,v}Z_{u,v}&= \lambda_{uv}^2 \quad \text{ for all } uv \in E_G\,.
\end{align*}
By the definition of $W_{u,v}$ and $Z_{u,v}$, the equations 
\begin{equation*}
	\sum_{i=0}^n W_{u_i, u_{i+1}} =0\,, \qquad	\sum_{i=0}^n Z_{u_i, u_{i+1}} =0
\end{equation*}
hold for every cycle $(u_0,u_1, \dots ,u_n, u_{n+1}=u_0)$ in $G$. 
Recall that the valuation of a product is the sum of valuations and the valuation of a sum is the minimum of valuations.
A consequence is that if a sum of functions equals zero,
then there are at least two summands with the minimal valuation.
Since we consider only valuations trivial on~$\CC$,
$\nu(W_{u,v})=\nu(-W_{v,u})$ for a valuation $\nu$.
These, together with Chevalley's theorem (see~\cite{Deuring}),
are the main ingredients for one implication of the following theorem 
that was proven in~\cite{flexibleLabelings}.
\begin{thm}\label{thm:nacflexible}
	A connected graph~$G$ with at least one edge has a flexible labeling iff it has a NAC-coloring.
\end{thm}
Actually, the following statement can be deduced from the proof of Theorem~\ref{thm:nacflexible}
with only minor modification --- replacing $0$ by $\alpha$.
This theorem explains how the functions~$W_{u,v}$ and $Z_{u,v}$ yield a NAC-coloring.
\begin{thm}\label{thm:valuationGivesNAC}
	Let $\lambda$ be a flexible labeling of a graph $G$. 
	Let $F(\C)$ be the complex function field of an algebraic motion $\C$ of $(G,\lambda)$.
	If $\alpha\in\QQ$ and $\nu$ is a valuation of $F(\C)$ 
	such that there exists edges $\bar{u}\bar{v},\widehat{u}\widehat{v}$ in $E_G$
	with $\nu(W_{\bar{u}\bar{v}})=\alpha$ and $\nu(W_{\widehat{u}\widehat{v}})>\alpha$, then $\delta: E_G\rightarrow\{\red{},\blue{}\}$ given by 
	\begin{equation}\label{eq:deltacol}
	  \begin{aligned}
			\delta(uv)=\red{} & \iff \nu(W_{u,v})>\alpha\,, \\
			\delta(uv)=\blue{} & \iff \nu(W_{u,v})\leq \alpha\,.
		\end{aligned}
	\end{equation}
	is a NAC-coloring.
\end{thm}

This motivates the assignment of some NAC-colorings to an algebraic motion.

\begin{defn}
  Let $\C$ be an algebraic motion of $(G,\lambda)$.
  A NAC-coloring $\delta\in\nac{G}$ is called \emph{active w.r.t.\ $\C$} if 
  there exists a valuation $\nu$ of $F(\C)$ and $\alpha\in\QQ$ such that
  \eqref{eq:deltacol} holds.
  The set of all active NAC-colorings of $G$ w.r.t.~$\C$ is denoted by $\nacC{G}{\C}$.
\end{defn}
For illustration, we compute the active NAC-colorings of the non-degenerated algebraic motion of a deltoid.

\begin{exmp}	\label{ex:deltoid}
	Let $Q$ be a 4-cycle with a labeling $\lambda$ given by $\lambda(\{1,2\})=\lambda(\{1,4\})=1$
	and $\lambda(\{2,3\})=\lambda(\{3,4\})=3$.
	There is an algebraic motion $\C$ of $(Q,\lambda)$ that can be parametrized by 
	\begin{align*}
		\rho_t(1) &= \left(0,0\right)\,,  & \rho_t(2)&= \left(1,0\right)\,, \\
		\rho_t(3)&= \left(\frac{4 {\left(t^{2} - 2\right)}}{t^{2} + 4},\frac{12  t}{t^{2} + 4}\right)\,, &
		\rho_t(4)&= \left(\frac{t^{4} - 13  t^{2} + 4}{t^{4} + 5  t^{2} + 4},\,
				\frac{6 t{\left(t^{2} - 2 \right)}}{t^{4} + 5  t^{2} + 4}\right)
	\end{align*}
	for $t\in\RR$. Now, we have
	\begin{equation*}
		W_{1, 2}=1 \,,\:
		W_{2, 3}=\frac{ 3 (t + 2 \ci) }{ t - 2 \ci }\,,\:
		W_{3, 4}=\frac{ -3(t + \ci) }{ t - \ci }\,,\:
		W_{4, 1}=\frac{ -(t + \ci)(t + 2 \ci) }{ (t - \ci)(t - 2 \ci) }\,.
	\end{equation*}
	Hence, the only non-trivial valuations correspond to the polynomials $t\pm \ci$ and~$t\pm 2\ci$.
	They give two pairs of conjugated NAC-colorings 
	by taking a suitable threshold $\alpha\in\{-1,0\}$, see Table~\ref{tab:deltoid} and Figure~\ref{fig:deltoidActiveNACs}.
	We remark that $|\nacC{Q}{\C}|=4$, whereas $|\nac{Q}|=6$. 
	The two non-active NAC-colorings correspond to the degenerated motion of $(Q,\lambda)$,
	where the vertices $2$ and $4$ coincide.
	
	\begin{table}[htb]
	\centering
		\begin{tabular}{c|c|cccc|cccc}
			edge & $\lambda$ & $\nu_{t+\ci}$ & $\delta_1$ & $\nu_{t-\ci}$ & $\conj{\delta_1}$ 
					& $\nu_{t+2\ci}$ & $\delta_2$ & $\nu_{t-2\ci}$ & $\conj{\delta_2}$ \\ \hline
			$\left\{1, 2\right\}$ & 1\rule{0em}{1.1em} & $0$ & \blue{} & $0$ & \red{}  & $0$ & \blue{} & $0$ & \red{}  \\
			$\left\{2, 3\right\}$ & 3 & $0$ & \blue{} & $0$ & \red{}  & $1$ & \red{}  & $-1$ & \blue{} \\
			$\left\{3, 4\right\}$ & 3 & $1$ & \red{}  & $-1$ & \blue{} & $0$ & \blue{} & $0$ & \red{}  \\
			$\left\{1, 4\right\}$ & 1 & $1$ & \red{}  & $-1$ & \blue{} & $1$ & \red{}  & $-1$ & \blue{} \\
		\end{tabular}
	\caption{Valuations giving the active NAC-colorings of a deltoid}
	\label{tab:deltoid}
	\end{table}
	\begin{figure}[htb]
		\centering
		\begin{tabular}{cccc}
			\begin{tikzpicture}[scale=0.75]
				\node[lnode] (1) at (0, 0) { 1 };
				\node[lnode] (2) at (1, 0) { 2 };
				\node[lnode] (3) at (292/97, 216/97) { 3 };
				\node[lnode] (4) at (2413/8245, 7884/8245) { 4 };
				\draw[bedge] (1)edge(2)  (2)edge(3);
				\draw[redge] (3)edge(4) (1)edge(4);
			\end{tikzpicture}
			&
			\begin{tikzpicture}[scale=0.75]
				\node[lnode] (1) at (0, 0) { 1 };
				\node[lnode] (2) at (1, 0) { 2 };
				\node[lnode] (3) at (292/97, 216/97) { 3 };
				\node[lnode] (4) at (2413/8245, 7884/8245) { 4 };
				\draw[bedge] (3)edge(4) (1)edge(4);
				\draw[redge] (1)edge(2)  (2)edge(3);
			\end{tikzpicture}
			&
			\begin{tikzpicture}[scale=0.75]
				\node[lnode] (1) at (0, 0) { 1 };
				\node[lnode] (2) at (1, 0) { 2 };
				\node[lnode] (3) at (292/97, 216/97) { 3 };
				\node[lnode] (4) at (2413/8245, 7884/8245) { 4 };
				\draw[bedge] (1)edge(2) (3)edge(4);
				\draw[redge] (2)edge(3) (1)edge(4);
			\end{tikzpicture}
			&
			\begin{tikzpicture}[scale=0.75]
				\node[lnode] (1) at (0, 0) { 1 };
				\node[lnode] (2) at (1, 0) { 2 };
				\node[lnode] (3) at (292/97, 216/97) { 3 };
				\node[lnode] (4) at (2413/8245, 7884/8245) { 4 };
				\draw[bedge] (2)edge(3)  (1)edge(4);
				\draw[redge] (1)edge(2) (3)edge(4);
			\end{tikzpicture}
			\\
			$\delta_1$ & $\conj{\delta_1}$ & $\delta_2$ & $\conj{\delta_2}$
		\end{tabular}
		\caption{Active NAC-colorings of a deltoid}
		\label{fig:deltoidActiveNACs}
	\end{figure}
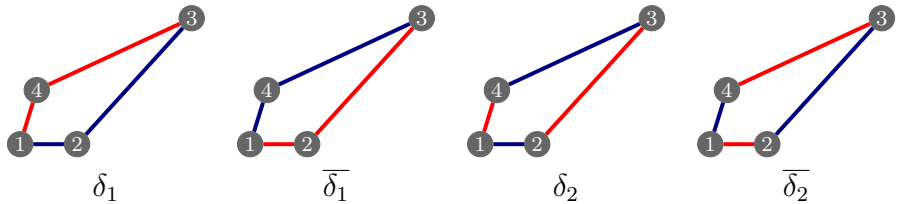
\end{exmp}

We conclude this section by three technical lemmas, 
which show that the active NAC-colorings do not depend on the choice of the fixed edge and 
that conjugated NAC-colorings are either both active or both non-active.
The first lemma says that the function field does not depend on the choice of the edge.
\begin{lem}\label{lem:motionFixingGivenEdge}
	Let $\lambda$ be a flexible labeling of $G$.
	Let $\C_{\bar{u},\bar{v}}$ be an algebraic motion of $(G,\lambda)$ w.r.t.\ an edge $\bar{u}\bar{v}$.
	If $u'v'\in E_G$ and $\varphi_{u'\!,v'}:\setreal{G}{\lambda}\rightarrow \setreal{G}{\lambda}$ is given by 
	\begin{align*}
		(x_w,y_w)_{w\in V_G} \mapsto  \bigg(&\frac{(x_w-x_{u'})(x_{v'}-x_{u'})+(y_w-y_{u'})(y_{v'}-y_{u'})}{\lambda(u'v')}, \\
												&\frac{(y_w-y_{u'})(x_{v'}-x_{u'})-(x_w-x_{u'})(y_{v'}-y_{u'})}{\lambda(u'v')}\bigg)_{w\in V_G}\,,
	\end{align*}
	then $\C_{u'\!,v'}=\varphi_{u'\!,v'}(\C_{\bar{u},\bar{v}})$ is an algebraic motion of $(G,\lambda)$ w.r.t.\ an edge $u'v'$
	and $\varphi_{u'\!,v'}:\C_{\bar{u},\bar{v}} \rightarrow \C_{u'\!,v'}$ is birational.
\end{lem}
\begin{proof}
	By direct computation, one can check that $u'v'$ is indeed fixed in~$\C_{u'\!,v'}$  and that all realizations in $\C_{u'\!,v'}$ are compatible with $\lambda$.
	The rational inverse of $\varphi_{u'\!,v'}$ is $\varphi_{\bar{u},\bar{v}}$.
\end{proof}

The following lemma shows that active NAC-colorings are independent of the choice of the fixed edge.
\begin{lem}\label{lem:activeNACindependentOnFixedEdge}
	Let $G$ be a graph with a flexible labeling $\lambda$.
	If $\C_{u'\!,v'}$ and $\C_{\bar{u},\bar{v}}$ are as in Lemma~\ref{lem:motionFixingGivenEdge}, 
	then $\nacC{G}{\C_{u'\!,v'}}=\nacC{G}{\C_{\bar{u},\bar{v}}}$.
\end{lem}
\begin{proof}
	Let $\delta \in \nacC{G}{\C_{\bar{u},\bar{v}}}$, i.e., there exists a valuation $\bar{\nu}$ of $F(\C_{\bar{u},\bar{v}})$ and $\alpha\in\QQ$ such that
	$\delta(uv)=\red{} \iff \bar{\nu}(W^{\bar{u}\bar{v}}_{u,v})>\alpha$ for all $uv\in E_G$.	
	Let $\varphi_{u'\!,v'}:\C_{\bar{u},\bar{v}} \rightarrow \C_{u'\!,v'}$ be the birational map from Lemma~\ref{lem:motionFixingGivenEdge}.
	Hence, there is a function field isomorphism $\phi:F(\C_{u'\!,v'})\rightarrow F(\C_{\bar{u},\bar{v}})$ given by $f\mapsto f\circ \varphi_{u'\!,v'}$.
	We define a valuation $\nu'$ of $F(\C_{u'\!,v'})$ by $\nu'(f):=\bar{\nu}(\phi(f))$.
	If $W^{u'\!v'}_{u,v}\in F(\C_{u'\!,v'})$, then
	\begin{align*}
		W^{u'\!v'}_{u,v} \circ \varphi_{u'\!,v'} &= 
		 \Big((x_v-x_u)(x_{v'}-x_{u'})+(y_v-y_u)(y_{v'}-y_{u'}) \Big)/\lambda(u'v')\\
				& \quad 	+ \ci \Big((y_v-y_u)(x_{v'}-x_{u'})-(x_v-x_u)(y_{v'}-y_{u'})\Big)/\lambda(u'v') 	\\
			&=\Big((x_v-x_u)+\ci(y_v-y_u) \Big)\cdot\Big((x_{v'}-x_{u'})-\ci (y_{v'}-y_{u'})\Big)/\lambda(u'v')\,.
	\end{align*}
	Therefore, $\nu'(W^{u'\!v'}_{u,v}) = \bar{\nu}(W^{\bar{u}\bar{v}}_{u,v}) + \bar{\nu}(Z^{\bar{u}\bar{v}}_{u'\!,v'})$.
	This concludes the proof, since  
	\begin{equation*}
		\delta(uv)=\red{} \iff \bar{\nu}(W^{\bar{u}\bar{v}}_{u,v})>\alpha \iff \nu'(W^{u'\!v'}_{u,v})> \alpha+ \bar{\nu}(Z^{\bar{u}\bar{v}}_{u'\!,v'})\,.
	\end{equation*}		
\end{proof}
Finally, we show that the set of active NAC-colorings is closed under conjugation.
\begin{lem}	\label{lem:conjugatedActiveNACs}
	Let $\lambda$ be a flexible labeling of a graph $G$. 
	Let $\C$ be an algebraic motion of~$(G,\lambda)$.
	If $\delta, \conj{\delta} \in \nac{G}$ are conjugated,
	then $\delta \in \nacC{G}{\C}$ if and only if $\conj{\delta} \in \nacC{G}{\C}$.
\end{lem}
\begin{proof}
	Let $\delta$ be an active NAC-coloring of $G$ w.r.t.\ $\C$ given by a valuation~$\nu$ of $F(\C)$ and a threshold $\alpha$.
	Since the algebraic motion $\C$ is a real algebraic curve, it has complex conjugation defined on its complex points.
	This induces another valuation $\conj{\nu}$ of $F(\C)$ given by
	$\conj{\nu}(f):= \nu(\conj{f})$ for any $f\in F(\C)$,
	where $\conj{f}$ is given by $\conj{f}(\rho):=\conj{f(\overline{\rho})}$ for every~$\rho\in\C$.
	If $\beta=\max\{\nu(Z_{u,v}) \colon \nu(Z_{u,v})< -\alpha,  uv \in E_G\}$, 
	then $\conj{\nu}$ and $\beta$ satisfy \eqref{eq:deltacol} for $\conj{\delta}$, since for every edge $uv\in E_G$:
	\begin{align*}
		\delta(uv)=\red{} &\iff \alpha < \nu(W_{u,v}) \iff -\alpha >\nu(Z_{u,v})\\
			&\iff \beta \geq\nu(Z_{u,v}) \iff \beta \geq \conj{\nu}(W_{u,v}) \iff \conj{\delta}(uv)=\blue{}\,.
	\end{align*}
	Namely, $\conj{\delta}$ is in $\nacC{G}{\C}$.
\end{proof}

\section{Combinatorial tools}	\label{sec:combinatorialCondition}
From now on, we are interested only in proper flexible labelings,
namely, the question, whether a graph is movable.
One of our main tools is introduced in this section: 
An edge $uv$ can be added to a graph without changing its algebraic motion,
if the vertices~$u$ and $v$ are connected by a path that is monochromatic in every active NAC-coloring.
This leads to the notion of constant distance closure ---
augmenting the graph by edges with the property above,
taking into account all NAC-colorings of the graph instead of active ones.
Hence, we obtain a necessary combinatorial condition on movability:
a graph can be movable only if its constant distance closure has a NAC-coloring.
Based on this necessary condition, we show that so called tree-decomposable graphs are not movable.
At the end of the section, we list all maximal constant distance closures 
of graphs up to~8 vertices having a spanning Laman graph
that satisfy the necessary condition.

The following statement guarantees that adding an edge $uv$ 
with the mentioned property preserves an algebraic motion,
since the distance between $u$ and $v$ is constant during the motion.
\begin{lem}
	\label{lem:constantDistance}
	Let $G$ be a graph, $\lambda$ a flexible labeling of $G$ and $u,v\in V_G$ where $uv\not\in E_G$.
	Let $\C$ be an algebraic motion of $(G,\lambda)$
	such that $\forall \rho \in \C: \rho(u)\neq\rho(v)$.
	If there exists a $uv$-path $P$ in $G$ such that $P$ is monochromatic for all $\delta \in \nacC{G}{\C}$,
	then $\lambda$ has a unique extension $\lambda'$ of $G'=(V_G, E_G\cup\{uv\})$,
	such that $\C$ is an algebraic motion of $(G',\lambda')$ and
	$\nacC{G'}{\C}=\{\delta'\in\nac{G'} \colon \delta'|_{E_G}\in\nacC{G}{\C}\}$.
\end{lem}
\begin{proof}
	Let $S=\{\norm{\rho(u)-\rho(v)}\in\RR_+\colon \rho \in \C\}$.
	We first show that $S$ is finite.
	By Lemma~\ref{lem:activeNACindependentOnFixedEdge},
	we can assume that the first edge $u_1u_2$ of the path $P$ is the fixed one in $\C$.
	If there is any $u_ku_{k+1}$ in $P$ such that $W_{u_k,u_{k+1}}$ is transcendental,
	then there is a valuation~$\nu$ such that $\nu(W_{u_k,u_{k+1}})>0$ by Chevalley's Theorem (see \cite{Deuring}).
	Hence, an active NAC-coloring can be constructed by Theorem~\ref{thm:valuationGivesNAC} with $\nu(W_{u_1,u_2})=0$,
	which contradicts that $P$ is monochromatic.
	Therefore, $W_{u_k,u_{k+1}}$ is algebraic for all $u_ku_{k+1}$ in $P$.
	Then there are only finitely many values for~$W_{u_k,u_{k+1}}$.
	These values correspond to possible angles of the line given by the realization of the vertices $u_k$ and $u_{k+1}$.
	Hence, there can only be finitely many elements in $S$.
	
	Indeed, we can show that $|S|\, =1$.
	Assume $S=\{s_1,\dots, s_\ell\}$, then $\C=\bigcup_{i\in\{1,\dots,\ell\}} \{\rho\in \C\colon \norm{\rho(u)-\rho(v)}^2=s_i^2\}$.
	Since $\C$ is irreducible, then $\ell=1$.
	We define $\lambda'$ by $\lambda'|_{E_G}=\lambda$ and $\lambda'(uv)=s_1$.
	
	The restriction of any active NAC-coloring $\delta'\!\in\nacC{G'}{\C}$ to $E_G$ is clearly in $\nacC{G}{\C}$.
	On the other hand, every active NAC-coloring of $G$ 
	is extended uniquely to an active NAC-coloring of $G'$, since the path~$P$ is monochromatic.
\end{proof}
Notice that it is sufficient to check the assumption only for non-conjugated active NAC-colorings 
due to Lemma~\ref{lem:conjugatedActiveNACs}.

Removal of an edge also preserves movability,
since edge lengths are assumed to be positive.
Together with the fact that $\nacC{G}{\C}\subset\nac{G}$ for any algebraic motion~$\C$,
this gives the following corollary.
\begin{cor}
	\label{cor:constantDistance}
	Let $G$ be a graph and $u,v\in V_G$ be such that $uv\notin E_G$.
	If there exists a $uv$-path $P$ in $G$ such that $P$ is monochromatic for all $\delta \in \nac{G}$,
	then $G$ is movable if and only if $G'=(V_G, E_G\cup\{uv\})$ is movable.
\end{cor}
\begin{proof}
  Let $\lambda'$ be a proper flexible labeling of $G'$.
  Clearly, $\lambda=\lambda'|_{E_G}$ is a flexible labeling of $G$. 
	A realization $\rho$ of $G'$ compatible with $\lambda'$ maps $u$ and $v$ to distinct points,
	since $\norm{\rho(u)-\rho(v)}=\lambda'(uv)\neq 0$.
	Clearly, $\rho$ is also a realization of $G$ and it is compatible with $\lambda$.
	The other direction follows from Lemma~\ref{lem:constantDistance}.
\end{proof}

Let us point out that there is no specific algebraic motion assumed in the previous corollary.
Hence, it can be used for proving that a graph is not movable in purely combinatorial way.
This is demonstrated by the following example.
\begin{exmp}	\label{ex:notMovable}
	The graph $G$ in Figure~\ref{fig:noProper7vertexGraph} is not movable:
	since the vertices 1 and 4 are connected by the path $(1,3,4)$ which is monochromatic in every NAC-coloring, and similarly for~2 and 5 with the path $(2,3,5)$,
	$G$ is movable if and only if $G'=(V_G,E_G\cup\{\{1,4\},\{2,5\}\})$ is movable.
	But $G'$ has no flexible labeling by Theorem~\ref{thm:nacflexible}, since it has no NAC-coloring.
	\begin{figure}[htb]
		\begin{center}
		    \begin{tikzpicture}[scale=0.65]
		    	\begin{scope}
				      \node[lnode] (1) at (-1,-1) {1};
				      \node[lnode] (2) at (1,-1) {2};
				      \node[lnode] (3) at (0,0) {3};
				      \node[lnode] (4) at (-1,1) {4};
				      \node[lnode] (5) at (1,1) {5};
				      \node[lnode] (6) at (-2.5,2) {6};
				      \node[lnode] (7) at (2.5,2) {7};
				      \draw[bedge] (1)edge(2) (1)edge(3) (2)edge(3)  (3)edge(5) (3)edge(4) (4)edge(5)  (5)edge(7) (2)edge(7);
				      \draw[redge] (1)edge(6) (4)edge(6) (6)edge(7);
		      	\end{scope}
		      	\begin{scope}[xshift=6.5cm]
				      \node[lnode] (1) at (-1,-1) {1};
				      \node[lnode] (2) at (1,-1) {2};
				      \node[lnode] (3) at (0,0) {3};
				      \node[lnode] (4) at (-1,1) {4};
				      \node[lnode] (5) at (1,1) {5};
				      \node[lnode] (6) at (-2.5,2) {6};
				      \node[lnode] (7) at (2.5,2) {7};
				      \draw[bedge] (1)edge(6) (4)edge(6) (1)edge(2) (1)edge(3) (2)edge(3)  (3)edge(5) (3)edge(4) (4)edge(5) ;
				      \draw[redge] (6)edge(7) (5)edge(7) (2)edge(7);
		      	\end{scope}
		      	\begin{scope}[xshift=13cm]
				      \node[lnode] (1) at (-1,-1) {1};
				      \node[lnode] (2) at (1,-1) {2};
				      \node[lnode] (3) at (0,0) {3};
				      \node[lnode] (4) at (-1,1) {4};
				      \node[lnode] (5) at (1,1) {5};
				      \node[lnode] (6) at (-2.5,2) {6};
				      \node[lnode] (7) at (2.5,2) {7};
				      \draw[bedge] (6)edge(7)  (1)edge(2) (1)edge(3) (2)edge(3)  (3)edge(5) (3)edge(4) (4)edge(5) ;
				      \draw[redge] (1)edge(6) (4)edge(6) (5)edge(7) (2)edge(7);
		      	\end{scope}
		    \end{tikzpicture}
		\caption{All non-conjugated NAC-colorings of a Laman graph with no proper flexible labeling}
		\label{fig:noProper7vertexGraph}
		\end{center}
	\end{figure}
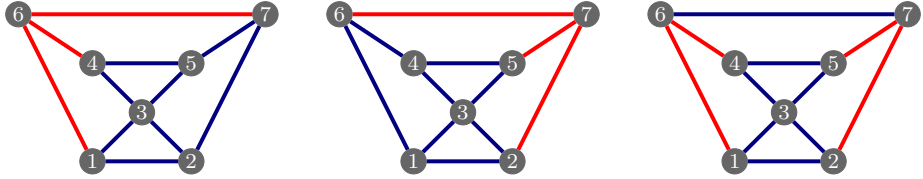
\end{exmp}

The corollary and example motivate the next definition. 
The name is inspired by the constant distance between vertices $u$ and $v$
 in Lemma~\ref{lem:constantDistance} during the motion.

\begin{defn}
	Let $G$ be a graph. 
	Let $\upairs{G}$ denote the set of all pairs $\{u,v\}\subset V_G$ such that $uv\notin E_G$ and 
	there exists a path from $u$ to $v$ which is monochromatic for all $\delta \in \nac{G}$.	
	If there exists a sequence of graphs $G_0, \dots, G_n$ such that
	\begin{enumerate}
		\item $G=G_0$,
		\item $G_i=(V_{G_{i-1}},E_{G_{i-1}} \cup \upairs{G_{i-1}})$ for $i\in\{1,\dots,n\}$,
		\item $\upairs{G_n}=\emptyset$,
	\end{enumerate}
	then the graph $G_n$ is called \emph{the constant distance closure of $G$}, denoted by $\cdc{G}$.
\end{defn}

The idea of repetitive augmenting the graph $G$ by edges in $\upairs{G}$ is 
to decrease the number of inequalities checking injectivity of compatible realizations ---
adjacent vertices must be always mapped to different points.
This can be seen in Example~\ref{ex:notMovable} --- the construction of a flexible labeling 
from a NAC-coloring described in~\cite{flexibleLabelings} always coincides the vertices $1$ and $4$, or $2$ and $5$.
But the edges $\{1,4\}$ and $\{2,5\}$ are added to the constant distance closure already in the first iteration.

In other words, considering the constant distance closure, while seeking for a proper flexible labeling,
utilizes more information from the graph than taking the graph itself.
For instance, since $\nac{\cdc{G}}\subset\nac{G}$, 
the NAC-colorings of $G$ that are active only for motions with non-injective realizations might be eliminated.
This is summarized by the following statement.
\begin{thm}	\label{thm:constDistClosureMovable}
	A graph $G$ is movable if and only if the constant distance closure of $G$ is movable.
\end{thm}
\begin{proof}
	The theorem follows by recursive application of Corollary~\ref{cor:constantDistance}.
\end{proof}
An immediate consequence is that a graph $G$ can be movable
only if the constant distance closure $\cdc{G}$ has a flexible labeling,
namely, we relax the requirement on the labeling to be proper.
By Theorem~\ref{thm:nacflexible}, it is equivalent to say that if $G$ is movable,
then $\cdc{G}$ has a NAC-coloring.
We can reformulate this necessary condition using the next two lemmas.

\begin{lem} \label{lem:constDistClosureSubgraph}
	Let $G$ be a graph.
	If $H$ is a subgraph of $G$,
	then the constant distance closure $\cdc{H}$ is a subgraph of the constant distance closure $\cdc{G}$.
\end{lem}
\begin{proof}
	If we show that $\upairs{H}\subset \upairs{G}$, then the claim follows by induction.
	Let a non\-edge~$uv$ be in $\upairs{H}$, namely, there exists a path $P$ from $u$ to $v$ such that
	it is monochromatic for all NAC-colorings in $\nac{H}$. 
	But then $uv$ is also in $\upairs{G}$, since the path $P$ is monochromatic also for all $\delta\in\nac{G}$,
	because $P$ is a subgraph of $H$ and either $\delta|_{E_H}\in\nac{H}$ or $|\delta(E_H)|=1$.
\end{proof}
Now, we can show that having a NAC-coloring and being non-complete is the same for a constant distance closure.

\begin{lem} \label{lem:constDistClosureComplete}
	Let $G$ be a graph. The constant distance closure $\cdc{G}$ is the complete graph if 
	and only if there exists a spanning subgraph of $\cdc{G}$ that has no NAC-coloring.
\end{lem}
\begin{proof}
	If $\cdc{G}$ is the complete graph, then it has clearly no NAC-coloring.
	For the opposite implication,
	assume that there is a spanning subgraph $H$ of $\cdc{G}$ that has no NAC-coloring.
	Trivially, $\upairs{H}$ consists of all nonedges of~$H$.
	Hence, the constant distance closure of $H$ is the complete graph.
	By Lemma~\ref{lem:constDistClosureSubgraph}, $\cdc{G}$ is also complete.
\end{proof}

The previous statement clarifies that the necessary condition obtained from Theorem~\ref{thm:constDistClosureMovable}
can be expressed as follows
by relaxing the requirement of a flexible labeling being proper.
\begin{cor}\label{cor:necesarryCondProper}
	Let $G$ be a graph.
	If the constant distance closure $\cdc{G}$ is the complete graph, then $G$ is not movable.
\end{cor}

Let us use this necessary condition to prove that a certain class of Laman graphs is not movable.
We would like to thank Meera Sitharam for pointing us to this class.
\begin{defn}	\label{def:treeDecomposable}
	A graph $G$ is \emph{tree-decomposable} if it is a single edge, or
	there are three tree-decomposable subgraphs $H_1,H_2$ and $H_3$ of $G$ such that
	$V_G=V_{H_1}\cup V_{H_2} \cup V_{H_3}$, $E_G=E_{H_1}\cup E_{H_2} \cup E_{H_3}$
	and $V_{H_1}\cap V_{H_2}=\{u\}, V_{H_2}\cap V_{H_3}=\{v\}$ 
	and $ V_{H_1}\cap V_{H_3}=\{w\}$ for three distinct vertices $u,v,w\in V_G$.
\end{defn}
One could prove geometrically that the tree-decomposable graphs are not movable,
but the notion of constant distance closure allows to do it in a combinatorial way.
\begin{thm}\label{thm:treeDecomposable}
	If a graph is tree-decomposable, then it is not movable.
\end{thm}
\begin{proof}
	Let $G$ be a tree-decomposable graph. 
	It is sufficient to show that the constant
	distance closure $\cdc{G}$ is the complete graph and use Corollary~\ref{cor:necesarryCondProper}.
	We proceed by induction on the tree-decomposable construction.
	Clearly, the constant distance closure of a single edge is the edge itself which is $K_2$.
	Let $H_1,H_2$ and $H_3$ be tree-decomposable subgraphs of $G$ as in Definition~\ref{def:treeDecomposable},
	with the pairwise common vertices $u,v$ and $w$.
	By Lemma~\ref{lem:constDistClosureSubgraph} and induction assumption, 
	the subgraphs $H'_1,H'_2$ and $H'_3$ of $\cdc{G}$ induced by $V_{H_1}, V_{H_2}$ and $V_{H_3}$ respectively are complete.
	Thus, there is no NAC-coloring of $H'=(V_{H_1}\cup V_{H_2}\cup V_{H_3}, E_{H'_1}\cup E_{H'_2}\cup E_{H'_3})$, 
	since all edges in a complete graph must have the same color 
	and $H'_1,H'_2$ and $H'_3$ contain each an edge of the triangle induced by~$u,v,w$.
	By Lemma~\ref{lem:constDistClosureComplete}, $\cdc{G}$ is complete, since $H'$ is its spanning subgraph .
\end{proof}
We remark that the class of so called H1 graphs is a subset of tree-decomposable graphs,
hence, they are not movable. 
A graph is called H1 if it can be constructed from a single edge by a sequence of Henneberg~I steps ---
each step adds a new vertex by linking it to two existing ones.
The next statement recalls the known fact that Henneberg~I steps do not affect movability.

\begin{lem}	\label{lem:degTwo}
	Let $G$ be a graph and $u\in V_G$ be a vertex of degree two. The graph $G$ is movable if and only if $G'=G\setminus u$ is movable.
\end{lem}
\begin{proof}
	Let $v$ and $w$ be the neighbours of $u$.
	If $\lambda'$ is a proper flexible labeling of $G'$, then~$\lambda: E_G \rightarrow \RR_+$ given by $\lambda'|_{E_{G'}}=\lambda$ and
	$\lambda(uv)=\lambda(uw)=L$, where $L$ is the maximal distance between $v$ and $w$ in all realizations compatible with $\lambda'$,
	is a proper flexible labeling of $G$.
	On the other hand, the restriction of a proper flexible labeling of $G$ to~$G'$ is a proper flexible labeling,
	since there are only two possible points where $u$ can be placed if $v$ and $w$ are mapped to distinct points,
	 i.e., there must be infinitely many realization of~$G'$.
\end{proof}

Since the previous lemma justifies that the question of movability of a graph with vertices of degree two
reduces to a smaller graph with all degrees being different from two,
we can provide a list of ``interesting" graphs regarding movability.
By ``interesting", we mean, besides all vertices having degree at least three,
also the fact that they are spanned by a Laman graph. 
Recall that graphs that are not spanned by a Laman graph are clearly movable,
since a generic labeling is proper flexible.
We are interested in graphs that can be movable only due to a non-generic labeling.
So we conclude this section by the list of the ``interesting" constant distance closures up to 8 vertices:

\begin{thm}	\label{thm:listEightVertices}
	Let $G$ be a graph with at most $8$ vertices such that it has a spanning Laman subgraph
	and $\cdc{G}$ has no vertex of degree two.
	If $G$ satisfies the necessary condition of movability, i.e, 
	the constant distance closure $\cdc{G}$ is not complete, then~$\cdc{G}$ is 
	one of the graphs $K_{3,3}, K_{3,4}, K_{3,5}$, $K_{4,4}, L_1, \dots, L_6,  Q_1, \dots, Q_6,S_1, \dots, S_5$, or a spanning subgraph thereof,
	where the non-bipartite graphs are given by Figure~\ref{fig:constDistClosures}.
\end{thm}
\begin{proof}
	Using the list of Laman graphs \cite{listLamanGraphs}, 
	one can compute constant distance closures of all graphs spanned by a Laman graph with at most 8 vertices.	
	The computation shows that each constant distance closure is either a complete graph, 
	or it has a vertex of degree two,
	or it is a spanning subgraph (or the full graph) of one of 
	 $K_{3,3}, K_{3,4}, K_{3,5}, K_{4,4}$ or the graphs in Figure~\ref{fig:constDistClosures}.
\end{proof}

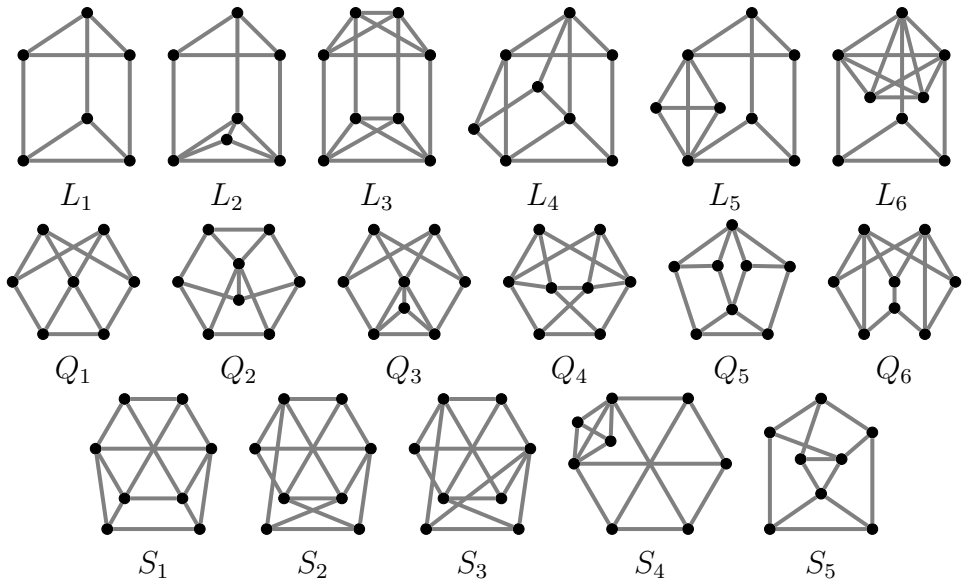
\begin{figure}[htb]
	\begin{center}
		\begin{tabular}{cccccc}
			\begin{tikzpicture}[scale=1.4]
			    \node[vertex] (0) at (0.60, 0.40) {};
			    \node[vertex] (1) at (1, 0) {};
			    \node[vertex] (2) at (1, 1) {};
			    \node[vertex] (3) at (0.60, 1.4) {};
			    \node[vertex] (6) at (0, 1) {};
			    \node[vertex] (7) at (0, 0) {};
			    \draw[edge] (0)edge(1) (0)edge(7) (1)edge(7) (2)edge(3) (2)edge(6) (3)edge(6)  ;
			    \draw[edge] (0)edge(3) (1)edge(2)  (6)edge(7)  ;
			\end{tikzpicture}
			&
			\begin{tikzpicture}[scale=1.4]
			    \node[vertex] (0) at (0.60, 0.40) {};
			    \node[vertex] (1) at (1, 0) {};
			    \node[vertex] (2) at (1, 1) {};
			    \node[vertex] (3) at (0.60, 1.4) {};
			    \node[vertex] (4) at (0.50, 0.2) {};
			    \node[vertex] (6) at (0, 1) {};
			    \node[vertex] (7) at (0, 0) {};
			    \draw[edge] (0)edge(1) (0)edge(7) (1)edge(7) (2)edge(3) (2)edge(6) (3)edge(6) (4)edge(0) (4)edge(1) (4)edge(7)  ;
			    \draw[edge] (0)edge(3) (1)edge(2)  (6)edge(7)  ;
			\end{tikzpicture}
	    &
			\begin{tikzpicture}[scale=1.4]
			    \node[vertex] (0) at (0.70, 0.40) {};
			    \node[vertex] (1) at (1, 1) {};
			    \node[vertex] (2) at (0.30, 1.4) {};
			    \node[vertex] (3) at (0, 1) {};
			    \node[vertex] (4) at (0, 0) {};
			    \node[vertex] (5) at (0.30, 0.40) {};
			    \node[vertex] (6) at (1, 0) {};
			    \node[vertex] (7) at (0.70, 1.4) {};
			    \draw[edge] (0)edge(4) (0)edge(5) (0)edge(6) (1)edge(2) (1)edge(3) (1)edge(7) (2)edge(3) (2)edge(7) (3)edge(7) (4)edge(5) (4)edge(6) (5)edge(6)  ;
			    \draw[edge] (0)edge(7) (1)edge(6) (2)edge(5) (3)edge(4)  ;
			\end{tikzpicture}
			&
			\begin{tikzpicture}[scale=1.4]
			    \node[vertex] (0) at (1, 0) {};
			    \node[vertex] (1) at (-0.30, 0.30) {};
			    \node[vertex] (2) at (0.30, 0.70) {};
			    \node[vertex] (3) at (1, 1) {};
			    \node[vertex] (4) at (0.60, 1.4) {};
			    \node[vertex] (5) at (0, 1) {};
			    \node[vertex] (6) at (0.60, 0.40) {};
			    \node[vertex] (7) at (0, 0) {};
			    \draw[edge] (0)edge(6) (0)edge(7) (1)edge(2) (3)edge(4) (3)edge(5) (4)edge(5) (6)edge(7)  ;
			    \draw[edge] (0)edge(3) (1)edge(5) (1)edge(7) (2)edge(4) (2)edge(6) (4)edge(6) (5)edge(7)  ;
			\end{tikzpicture}
			&
			\begin{tikzpicture}[scale=1.4]
			    \node[vertex] (0) at (0.60, 0.40) {};
			    \node[vertex] (1) at (1, 0) {};
			    \node[vertex] (2) at (1, 1) {};
			    \node[vertex] (3) at (0.60, 1.4) {};
			    \node[vertex] (4) at (0.30, 0.50) {};
			    \node[vertex] (5) at (-0.30, 0.50) {};
			    \node[vertex] (6) at (0, 1) {};
			    \node[vertex] (7) at (0, 0) {};
			    \draw[edge] (0)edge(1) (0)edge(7) (1)edge(7) (2)edge(3) (2)edge(6) (3)edge(6)  ;
			    \draw[edge] (0)edge(3) (1)edge(2) (4)edge(5) (4)edge(6) (4)edge(7) (5)edge(6) (5)edge(7) (6)edge(7)  ;
			\end{tikzpicture}
			&
			\begin{tikzpicture}[scale=1.4]
			    \node[vertex] (0) at (0.60, 0.40) {};
			    \node[vertex] (1) at (0, 0) {};
			    \node[vertex] (2) at (1, 0) {};
			    \node[vertex] (3) at (0.80, 0.60) {};
			    \node[vertex] (4) at (0.30, 0.60) {};
			    \node[vertex] (5) at (1, 1) {};
			    \node[vertex] (6) at (0, 1) {};
			    \node[vertex] (7) at (0.60, 1.4) {};
			    \draw[edge] (0)edge(1) (0)edge(2) (1)edge(2) (3)edge(4) (3)edge(5) (3)edge(6) (3)edge(7) (4)edge(5) (4)edge(6) (4)edge(7) (5)edge(6) (5)edge(7) (6)edge(7)  ;
			    \draw[edge] (0)edge(7) (1)edge(6) (2)edge(5)  ;
			\end{tikzpicture}
			\\
			$L_1$ & $L_2$ &$L_3$ &$L_4$ &$L_5$ &$L_6$ 
		\end{tabular}
	\begin{tabular}{cccccc}
		\begin{tikzpicture}[scale=1*0.8]
		      \node[vertex] (3) at (0.5, -0.866025) {};
		      \node[vertex] (0) at (1., 0.) {};
		      \node[vertex] (4) at (0.5, 0.866025) {};
		      \node[vertex] (5) at (-0.5, 0.866025) {};
		      \node[vertex] (1) at (-1.,  0.) {};
		      \node[vertex] (2) at (-0.5, -0.866025) {};
	   	      \node[vertex] (6) at (0,0) {};
		      \draw[edge] (2)edge(3)  (3)edge(6) (6)edge(2) (1)edge(4) (1)edge(5) (4)edge(0) (0)edge(5);
		      \draw[edge] (2)edge(1) (6)edge(4) (6)edge(5) (0)edge(3);
		\end{tikzpicture}
		&
		\begin{tikzpicture}[scale=1*0.8]
		    \node[vertex] (0) at (-1.00, 0.000) {};
		    \node[vertex] (1) at (1.00, 0.000) {};
		    \node[vertex] (2) at (0.500, 0.866) {};
		    \node[vertex] (3) at (0.500, -0.866) {};
		    \node[vertex] (4) at (-0.500, -0.866) {};
		    \node[vertex] (5) at (-0.500, 0.866) {};
		    \node[vertex] (6) at (0.000, -0.300) {};
		    \node[vertex] (7) at (0.000, 0.300) {};
		    \draw[edge] (1)edge(3) (6)edge(7) (0)edge(5)  ;
		    \draw[edge] (3)edge(4) (3)edge(7) (4)edge(7) (1)edge(6)  ;
		    \draw[edge] (0)edge(6) (5)edge(7) (2)edge(5) (2)edge(7)  ;
		    \draw[edge] (1)edge(2) (0)edge(4)  ;
		\end{tikzpicture}
		&
		\begin{tikzpicture}[scale=1*0.8]
		    \node[vertex] (1) at (0.500, 0.866) {};
		    \node[vertex] (0) at (-0.500, 0.866) {};
		    \node[vertex] (2) at (-1.00, 0.000) {};
		    \node[vertex] (3) at (1.00, 0.000) {};
		    \node[vertex] (4) at (0.000, -0.433) {};
		    \node[vertex] (5) at (0.500, -0.866) {};
		    \node[vertex] (6) at (-0.500, -0.866) {};
		    \node[vertex] (7) at (0.000, 0.000) {};
		    \draw[edge] (2)edge(6) (1)edge(7)  ;
		    \draw[edge] (0)edge(7) (3)edge(5)  ;
		    \draw[edge] (5)edge(7) (5)edge(6) (6)edge(7) (4)edge(6) (4)edge(7) (1)edge(2) (0)edge(3) (4)edge(5)  ;
		    \draw[edge] (1)edge(3) (0)edge(2)  ;
		\end{tikzpicture}		
		&
		\begin{tikzpicture}[scale=1*0.8]
		    \node[vertex] (0) at (0.500, -0.866) {};
		    \node[vertex] (1) at (-0.500, -0.866) {};
		    \node[vertex] (2) at (0.500, 0.866) {};
		    \node[vertex] (3) at (-0.500, 0.866) {};
		    \node[vertex] (4) at (-1.00, 0.000) {};
		    \node[vertex] (5) at (1.00, 0.000) {};
		    \node[vertex] (6) at (0.300, -0.100) {};
		    \node[vertex] (7) at (-0.300, -0.100) {};
		    \draw[edge] (0)edge(7) (1)edge(4)  ;
		    \draw[edge] (0)edge(5) (1)edge(6)  ;
		    \draw[edge] (3)edge(4) (3)edge(7) (4)edge(7) (5)edge(6) (2)edge(6) (2)edge(5) (0)edge(1)  ;
		    \draw[edge] (2)edge(4) (6)edge(7) (3)edge(5)  ;
		\end{tikzpicture}
		&
		\begin{tikzpicture}[scale=1*0.8]
		    \node[vertex] (0) at (-0.588, -0.809) {};
		    \node[vertex] (1) at (0.588, -0.809) {};
		    \node[vertex] (2) at (0.951, 0.309) {};
		    \node[vertex] (3) at (-0.951, 0.309) {};
		    \node[vertex] (4) at (0.235, 0.324) {};
		    \node[vertex] (5) at (-0.235, 0.324) {};
		    \node[vertex] (6) at (0.000, -0.400) {};
		    \node[vertex] (7) at (0.000, 1.00) {};
		    \draw[edge] (1)edge(2) (0)edge(3)  ;
		    \draw[edge] (0)edge(6) (1)edge(6) (0)edge(1)  ;
		    \draw[edge] (2)edge(4) (5)edge(6) (4)edge(7) (2)edge(7)  ;
		    \draw[edge] (5)edge(7) (3)edge(5) (3)edge(7) (4)edge(6)  ;
		\end{tikzpicture}
		&
		\begin{tikzpicture}[scale=1*0.8]
		    \node[vertex] (0) at (0.000, -0.433) {};
		    \node[vertex] (1) at (-0.500, -0.866) {};
		    \node[vertex] (2) at (0.500, -0.866) {};
		    \node[vertex] (3) at (1.00, 0.000) {};
		    \node[vertex] (4) at (-1.00, 0.000) {};
		    \node[vertex] (5) at (0.000, 0.000) {};
		    \node[vertex] (6) at (0.500, 0.866) {};
		    \node[vertex] (7) at (-0.500, 0.866) {};
		    \draw[edge] (3)edge(7) (4)edge(6)  ;
		    \draw[edge] (0)edge(5) (1)edge(7) (4)edge(7) (3)edge(6) (2)edge(3) (2)edge(6) (1)edge(4)  ;
		    \draw[edge] (5)edge(6) (0)edge(2)  ;
		    \draw[edge] (5)edge(7) (0)edge(1)  ;
		\end{tikzpicture}
		\\
			$Q_1$ & $Q_2$ &$Q_3$ &$Q_4$ &$Q_5$ &$Q_6$ 
	\end{tabular}
	\begin{tabular}{ccccc}
			\begin{tikzpicture}[scale=0.761]
			\node[vertex] (5) at (0.5, -0.866025) {};
			\node[vertex] (4) at (1., 0.) {};
			\node[vertex] (2) at (0.5, 0.866025) {};
			\node[vertex] (1) at (-0.5, 0.866025) {};
			\node[vertex] (7) at (-1.,  0.) {};
			\node[vertex] (6) at (-0.5, -0.866025) {};
			\node[vertex] (3) at (0.8,-1.4) {};
			\node[vertex] (0) at (-0.8,-1.4) {};
			\draw[edge] (7)edge(0) (3)edge(4);
			\draw[edge]  (6)edge(5) (5)edge(4) (7)edge(4) (7)edge(6) (0)edge(3);
			\draw[edge] (2)edge(4) (2)edge(6);
			\draw[edge] (1)edge(5) (7)edge(1) ;
			\draw[edge] (2)edge(1);
			\draw[edge] (0)edge(6)  (5)edge(3) ;
		\end{tikzpicture}
		&
		\begin{tikzpicture}[scale=0.761]
			\node[vertex] (1) at (-0.8,-1.4) {};
			\node[vertex] (5) at (1., 0.) {};
			\node[vertex] (3) at (-1.,  0.) {};
			\node[vertex] (6) at (-0.5, 0.866025) {};
			\node[vertex] (2) at (0.5, 0.866025) {};
			\node[vertex] (4) at (-0.5, -0.866025) {};
			\node[vertex] (7) at (0.5, -0.866025) {};
			\node[vertex] (0) at (0.8,-1.4) {};
			\draw[edge] (5)edge(3) (2)edge(6) (4)edge(0);
			\draw[edge]  (2)edge(4) (0)edge(1) (5)edge(7);
			\draw[edge] (2)edge(5) (7)edge(4) (3)edge(6);
			\draw[edge] (1)edge(6) (7)edge(6) (3)edge(4) (1)edge(7) (5)edge(0);
		\end{tikzpicture}
		&
		\begin{tikzpicture}[scale=0.761]
			\node[vertex] (1) at (-0.8,-1.4) {};
			\node[vertex] (7) at (1., 0.) {};
			\node[vertex] (3) at (-1.,  0.) {};
			\node[vertex] (6) at (-0.5, -0.866025) {};
			\node[vertex] (2) at (0.5, -0.866025) {};
			\node[vertex] (5) at (-0.5, 0.866025) {};
			\node[vertex] (4) at (0.5, 0.866025) {};
			\node[vertex] (0) at (0.8,-1.4) {};
			\draw[edge] (7)edge(3) (2)edge(6) (5)edge(1);
			\draw[edge]  (2)edge(5) (7)edge(4) (0)edge(6);
			\draw[edge] (2)edge(7) (4)edge(5) (3)edge(6);
			\draw[edge] (4)edge(6) (3)edge(5) (7)edge(1) (0)edge(7);
			\draw[edge] (0)edge(1);
		\end{tikzpicture}
		&
		\begin{tikzpicture}[scale=1]
			\node[vertex] (a1) at (0.5, -0.866025) {};
			\node[vertex] (a2) at (1., 0.) {};
			\node[vertex] (a3) at (0.5, 0.866025) {};
			\node[vertex] (a4) at (-0.5, 0.866025) {};
			\node[vertex] (a5) at (-1.,  0.) {};
			\node[vertex] (a6) at (-0.5, -0.866025) {};
			\pgfmathsetmacro\hundredFiftycos{cos(150)};
			\pgfmathsetmacro\hundredFiftysin{sin(150)};
			\node[vertex] (a7) at (0.6*\hundredFiftycos, 0.6*\hundredFiftysin) {};
			\node[vertex] (a8) at (1.1*\hundredFiftycos, 1.1*\hundredFiftysin) {};
			\draw[edge] (a1)edge(a4) (a2)edge(a3) (a3)edge(a6) (a6)edge(a5) (a2)edge(a5);
			\draw[edge] (a5)edge(a4) (a6)edge(a1) (a3)edge(a4) (a2)edge(a1);
			\draw[edge] (a5)edge(a7) (a5)edge(a8) (a7)edge(a4) (a7)edge(a8) (a4)edge(a8);
		\end{tikzpicture}

		&
		\begin{tikzpicture}[scale=1.146]
		    \node[vertex] (0) at (0.000, 0.700) {};
		    \node[vertex] (1) at (-0.588, -0.809) {};
		    \node[vertex] (2) at (0.588, -0.809) {};
		    \node[vertex] (3) at (0.588, 0.309) {};
		    \node[vertex] (4) at (-0.588, 0.309) {};
		    \node[vertex] (5) at (-0.235, 0) {};
		    \node[vertex] (6) at (0.000, -0.400) {};
		    \node[vertex] (7) at (0.235, 0) {};
		    \draw[edge] (0)edge(3) (0)edge(4) (0)edge(5) (1)edge(2) (1)edge(4) (1)edge(6) (2)edge(3) (2)edge(6) (3)edge(7) (4)edge(7) (5)edge(6) (5)edge(7) (6)edge(7)  ;
		\end{tikzpicture} \\
		$S_1$ & $S_2$ & $S_3$ & $S_4$ & $S_5$
	\end{tabular}
	\caption{Maximal non-bipartite constant distance closures of graphs with a spanning Laman subgraph, at most 8 vertices and no vertex of degree two.}
	\label{fig:constDistClosures}
	\end{center}
\end{figure}

\section{Construction of proper flexible labelings}
\label{sec:constructions}
The goal of this section is to prove that all graphs listed in Theorem~\ref{thm:listEightVertices} are actually movable.
By this we also show that the necessary condition of movability (the constant distance closure is non-complete) is 
also sufficient for graphs up to eight vertices.
Four general ways of constructing a proper flexible labeling are presented.
The first two are known --- the Dixon~I construction for bipartite graphs \cite{Dixon}
and the construction from a single NAC-coloring presented in~\cite{flexibleLabelings}.
We describe a new construction that produces an algebraic motion with two active NAC-colorings
based on a certain injective embedding of vertices in $\RR^3$.
The fourth method assumes two movable subgraphs
whose union spans the whole graph and whose motions coincide on the intersection.
We provide a proper flexible labeling for $S_5$ ad~hoc, since none of the four methods applies.
Animations for the movable graphs can be found in \cite{animations}.
In the conclusion, we give an example showing that the necessary condition 
is not sufficient for graphs with arbitrary number of vertices.
In order to be self-contained, we recall Dixon's construction.
\begin{lem}	\label{lem:bipartiteGraphs}
	Every bipartite graph with at least three vertices is movable. 
\end{lem}
\begin{proof}
	Let $(X,Y)$ be a bipartite partition of a graph $G$.
	Hence, a realization with the vertices of one partition set on the $x$-axis,
	and vertices of the other on the $y$-axis induces a proper flexible labeling by Dixon's construction~\cite{Dixon,Stachel}:
	\begin{equation*}
		\rho_t(v) = \begin{cases}	(\sqrt{x_v^2 - t^2},0)\,, &\text{ if } v \in X\,, \\
									(0,\sqrt{y_v^2 + t^2})\,, &\text{ if } v \in Y\,, \\
					\end{cases}
	\end{equation*}
	where $x_v,y_v$ are arbitrary nonzero real numbers.
	Let $\lambda(uv):=\sqrt{x_v^2 + y_v^2}$ for all $u\in X$ and $v\in Y$.
	By the Pythagorean Theorem, $\rho_t$ is compatible with $\lambda$ for every sufficiently small $t$.
\end{proof}

The following method from \cite{flexibleLabelings} was used in the proof of Theorem~\ref{thm:nacflexible}
but without the assumption guaranteeing injectivity of realizations.
\begin{lem}\label{lem:oneNac}
	Let $\delta$ be a NAC-coloring of a graph $G$.
	Let $R_1, \dots, R_m$, resp.\ $B_1, \dots, B_n$, be the sets of vertices of connected components 
	of the graph obtained from $G$ by keeping only \red{}, resp.\ \blue{}, edges.
	If $|R_i \cap B_j|\,\leq 1$ for all $i,j$, then $G$ is movable.
\end{lem}
\begin{proof}
	For $\alpha\in [0,2\pi)$, we define a realization $\rho_\alpha\colon V_G\rightarrow \RR^2$ by 
	\begin{equation*}
	  \rho_\alpha(v)=i\cdot (1,0) + j\cdot (\cos\alpha, \sin\alpha)\,,
	\end{equation*}
	where $i$ and $j$ are such that~$v\in R_i \cap B_j$.
	Now, the realization $\rho_\alpha$ is compatible 
	with the labeling $\lambda: E_G\rightarrow \RR_+$, given by $\lambda(uv)=\norm{\rho_{\pi/2}(u)-\rho_{\pi/2}(v)}$,
	for every $\alpha\in [0,2\pi)$.
	The induced flexible labeling $\lambda$ is proper, 
	since all realizations $\rho_\alpha$, $\alpha\notin\{0,\pi\}$, are injective by the assumption $|R_i \cap B_j|\,\leq 1$.
\end{proof}

The construction yields proper flexible labelings for $L_1, \dots, L_6$,
since there are NAC-colorings satisfying the assumption, see Figure~\ref{fig:graphsOneNAC}.
The displayed proper flexible labelings can be obtained
by the more general ``zikzag" construction from \cite{flexibleLabelings}.
\begin{cor}
	\label{cor:oneNacGraphs}
	The graphs $L_1, \dots, L_6$, see Figure~\ref{fig:constDistClosures} or~\ref{fig:graphsOneNAC}, are movable.
\end{cor}

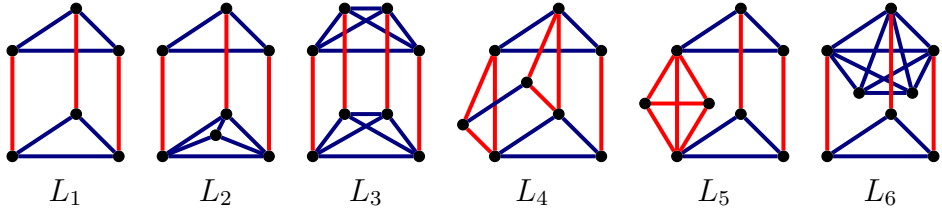
\begin{figure}[htb]
	\centering
	\begin{tabular}{cccccc}
	\begin{tikzpicture}[scale=1.4]
	    \node[vertex] (0) at (0.60, 0.40) {};
	    \node[vertex] (1) at (1, 0) {};
	    \node[vertex] (2) at (1, 1) {};
	    \node[vertex] (3) at (0.60, 1.4) {};
	    \node[vertex] (6) at (0, 1) {};
	    \node[vertex] (7) at (0, 0) {};
	    \draw[bedge] (0)edge(1) (0)edge(7) (1)edge(7) (2)edge(3) (2)edge(6) (3)edge(6)  ;
	    \draw[redge] (0)edge(3) (1)edge(2)  (6)edge(7)  ;
	\end{tikzpicture}
	&
	\begin{tikzpicture}[scale=1.4]
	    \node[vertex] (0) at (0.60, 0.40) {};
	    \node[vertex] (1) at (1, 0) {};
	    \node[vertex] (2) at (1, 1) {};
	    \node[vertex] (3) at (0.60, 1.4) {};
	    \node[vertex] (4) at (0.50, 0.2) {};
	    \node[vertex] (6) at (0, 1) {};
	    \node[vertex] (7) at (0, 0) {};
	    \draw[bedge] (0)edge(1) (0)edge(7) (1)edge(7) (2)edge(3) (2)edge(6) (3)edge(6) (4)edge(0) (4)edge(1) (4)edge(7)  ;
	    \draw[redge] (0)edge(3) (1)edge(2)  (6)edge(7)  ;
	\end{tikzpicture}
	&
	\begin{tikzpicture}[scale=1.4]
	    \node[vertex] (0) at (0.70, 0.40) {};
	    \node[vertex] (1) at (1, 1) {};
	    \node[vertex] (2) at (0.30, 1.4) {};
	    \node[vertex] (3) at (0, 1) {};
	    \node[vertex] (4) at (0, 0) {};
	    \node[vertex] (5) at (0.30, 0.40) {};
	    \node[vertex] (6) at (1, 0) {};
	    \node[vertex] (7) at (0.70, 1.4) {};
	    \draw[bedge] (0)edge(4) (0)edge(5) (0)edge(6) (1)edge(2) (1)edge(3) (1)edge(7) (2)edge(3) (2)edge(7) (3)edge(7) (4)edge(5) (4)edge(6) (5)edge(6)  ;
	    \draw[redge] (0)edge(7) (1)edge(6) (2)edge(5) (3)edge(4)  ;
	\end{tikzpicture}
	&
	\begin{tikzpicture}[scale=1.4]
	    \node[vertex] (0) at (1, 0) {};
	    \node[vertex] (1) at (-0.30, 0.30) {};
	    \node[vertex] (2) at (0.30, 0.70) {};
	    \node[vertex] (3) at (1, 1) {};
	    \node[vertex] (4) at (0.60, 1.4) {};
	    \node[vertex] (5) at (0, 1) {};
	    \node[vertex] (6) at (0.60, 0.40) {};
	    \node[vertex] (7) at (0, 0) {};
	    \draw[bedge] (0)edge(6) (0)edge(7) (3)edge(4) (3)edge(5) (4)edge(5) (6)edge(7)  ;
	    \draw[redge] (0)edge(3) (1)edge(5) (1)edge(7) (2)edge(4) (2)edge(6) (4)edge(6) (5)edge(7)  ;
	    \draw[bedge] (1)edge(2);
	\end{tikzpicture}
	&
	\begin{tikzpicture}[scale=1.4]
	    \node[vertex] (0) at (0.60, 0.40) {};
	    \node[vertex] (1) at (1, 0) {};
	    \node[vertex] (2) at (1, 1) {};
	    \node[vertex] (3) at (0.60, 1.4) {};
	    \node[vertex] (4) at (0.30, 0.50) {};
	    \node[vertex] (5) at (-0.30, 0.50) {};
	    \node[vertex] (6) at (0, 1) {};
	    \node[vertex] (7) at (0, 0) {};
	    \draw[bedge] (0)edge(1) (0)edge(7) (1)edge(7) (2)edge(3) (2)edge(6) (3)edge(6)  ;
	    \draw[redge] (0)edge(3) (1)edge(2) (4)edge(5) (4)edge(6) (4)edge(7) (5)edge(6) (5)edge(7) (6)edge(7)  ;
	\end{tikzpicture}
	&
	\begin{tikzpicture}[scale=1.4]
	    \node[vertex] (0) at (0.60, 0.40) {};
	    \node[vertex] (1) at (0, 0) {};
	    \node[vertex] (2) at (1, 0) {};
	    \node[vertex] (3) at (0.80, 0.60) {};
	    \node[vertex] (4) at (0.30, 0.60) {};
	    \node[vertex] (5) at (1, 1) {};
	    \node[vertex] (6) at (0, 1) {};
	    \node[vertex] (7) at (0.60, 1.4) {};
	    \draw[bedge] (0)edge(1) (0)edge(2) (1)edge(2) (3)edge(4) (3)edge(5) (3)edge(6) (3)edge(7) (4)edge(5) (4)edge(6) (4)edge(7) (5)edge(6) (5)edge(7) (6)edge(7)  ;
	    \draw[redge] (0)edge(7) (1)edge(6) (2)edge(5)  ;
	\end{tikzpicture}
	\\
	$L_1$ & $L_2$ &$L_3$ &$L_4$ &$L_5$ &$L_6$ 	
	\end{tabular}
	\caption{The NAC-colorings inducing a proper flexible labeling.}
	\label{fig:graphsOneNAC}
\end{figure}

Now, we present a construction assuming a special injective embedding in $\RR^3$.
The lemma also gives a hint, how existence of such an embedding can be checked (and an embedding found), 
if we know all NAC-colorings of the given graph.
\begin{lem}\label{lem:twoNacs}
Let $G=(V,E)$ be a graph with an injective embedding $\omega:V\rightarrow\RR^3$ such that for every edge 
$uv\in E$, the vector $\omega(u)-\omega(v)$ is parallel to one of the four vectors $(1,0,0)$, $(0,1,0)$, $(0,0,1)$, $(-1,-1,-1)$,
and all four directions are present.
Then~$G$ is movable.

Moreover, there exists an algebraic motion of $G$ with exactly two active NAC-colorings 
modulo conjugation. Two edges are parallel in the embedding~$\omega$ if and only if they
receive the same pair of colors in the two active NAC-colorings.
\end{lem}
\begin{proof}
Let $\rho_t:\{1,2,3,4\}\to\RR^2$ be a parametrization of an algebraic motion of the 4-cycle with a labeling $\lambda$.
We define three functions from $\RR$ to $\RR^2$ by 
\[ f_1(t)=\rho_t(2)-\rho_t(1),\ f_2(t)=\rho_t(3)-\rho_t(2),\ f_3(t)=\rho_t(4)-\rho_t(3) . \]
The norms $\norm{f_1}, \norm{f_2}, \norm{f_3}$ and $\norm{-f_1(t)-f_2(t)-f_3(t)}$ 
are the corresponding values of $\lambda$, i.e., they are independent of $t$. 
For each $t\in\RR$, we define 
\[ \widetilde{\rho_t}:V\to\RR^2,\  u\mapsto \omega_1(u)f_1(t)+\omega_2(u)f_2(t)+\omega_3(u)f_3(t)\,, \]
where $\omega(u)=\left(\omega_1(u),\omega_2(u),\omega_3(u)\right)$.
For any edge $uv$, $\widetilde{\rho_t}(u)-\widetilde{\rho_t}(v)$ is 
a multiple of $f_1,f_2,f_3$ or $-f_1(t)-f_2(t)-f_3(t)$ by assumption.
Thus, the distance $\norm{\widetilde{\rho_t}(u)-\widetilde{\rho_t}(v)}$ is independent of $t$ and different
from zero. 
Hence, the set of all $\rho_t$ is an algebraic motion; this proves the first statement.

In order to construct an algebraic motion with two active NAC-colorings,
we take~$\rho_t$ to be the parametrization of the deltoid in Example~\ref{ex:deltoid}.
For any edge $uv$, the function~$W_{u,v}$ is just a scalar multiple 
of one of the functions in the example.
Hence, there are only two active NAC-colorings modulo conjugation, see Table~\ref{tab:deltoid}.
\end{proof}

\begin{rem}
	Clearly, we can construct also an algebraic motion 
	with three non-conjugated active NAC-colorings by taking $\rho_t$ 
	from an algebraic motion of the 4-cycle with a general edge lengths,
	which has three non-conjugated active NAC-colorings.
	This also shows that if $\delta_1$ and $\delta_2$ are the two active NAC-colorings from the second statement of the lemma,
	then the coloring~$\delta_3$, given by $\delta_3(e)=\blue{}$ if and only if $\delta_1(e)=\delta_2(e)$,
	is also a NAC-coloring of $G$. 
	This follows from the fact that there are only three non-conjugated NAC-colorings of a 4-cycle with one chosen edge being always \blue{}
	and they are related as given above.
\end{rem}

Lemma~\ref{lem:twoNacs} allows to compute algebraic motions with exactly two active NAC-color\-ings:
For any pair of NAC-colorings, try to find an embedding $\omega:V\to\RR^3$ with edge directions
$(1,0,0)$, $(0,1,0)$, $(0,0,1)$ or $(-1,-1,-1)$ depending on the colors in these two colorings. This leads to a system
of linear equations. If it has a non-trivial solution, check if a general solution is injective.

\begin{exmp} \label{ex:Q1}
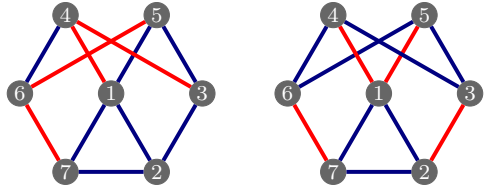
\begin{figure}[htb]
	\centering	
		\begin{tikzpicture}[scale=1.2]
		    \node[lnode] (5) at (0.500, 0.866) {5};
		    \node[lnode] (4) at (-0.500, 0.866) {4};
		    \node[lnode] (6) at (-1.00, 0.000) {6};
		    \node[lnode] (3) at (1.00, 0.000) {3};
		    \node[lnode] (2) at (0.500, -0.866) {2};
		    \node[lnode] (0) at (-0.500, -0.866) {7};
		    \node[lnode] (1) at (0.000, 0.000) {1};
		    \draw[bedge] (1)edge(2)  (1)edge(5)  (3)edge(5)  (2)edge(3)  (0)edge(2)  (4)edge(6)  (0)edge(1)    ;
		    \draw[redge] (0)edge(6)  (1)edge(4)  (3)edge(4)  (5)edge(6)   ;
		\end{tikzpicture}
		\hspace{0.5cm}
		\begin{tikzpicture}[scale=1.2]
		    \node[lnode] (5) at (0.500, 0.866) {5};
		    \node[lnode] (4) at (-0.500, 0.866) {4};
		    \node[lnode] (6) at (-1.00, 0.000) {6};
		    \node[lnode] (3) at (1.00, 0.000) {3};
		    \node[lnode] (2) at (0.500, -0.866) {2};
		    \node[lnode] (0) at (-0.500, -0.866) {7};
		    \node[lnode] (1) at (0.000, 0.000) {1};
		    \draw[redge] (1)edge(5);
		    \draw[bedge] (1)edge(2)  (3)edge(5)  (3)edge(4)  (0)edge(2)  (5)edge(6)  (4)edge(6)  (0)edge(1) ;
		    \draw[redge] (0)edge(6)  (1)edge(4)  (2)edge(3)   ;
		\end{tikzpicture}
\caption{The graph $Q_1$ with a pair of NAC-colorings giving an embedding in $\RR^3$.}
\label{fig:magic7}
\end{figure}
In order to find an embedding $\omega:V\to\RR^3$ for the graph $Q_1$ 
using the NAC-colorings in Figure~\ref{fig:magic7}, such that 
every edge colored with blue/blue is parallel to~$(1,0,0)$,
every edge colored with blue/red is parallel to $(0,1,0)$, 
every edge colored with red/blue is parallel to $(0,0,1)$, and
every edge colored with red/red is parallel to~$(-1,-1,-1)$, 
we put $\omega(1)$ to the origin and introduce variables $x_i,y_i,z_i$ for $p_i:=\omega(i)$, $i=2,\dots,7$.
For each edge, we have two linear equations. We obtain the system
\begin{align*}
	0&=y_{7}=z_{7}=y_{7} - y_{2}=z_{7} - z_{2}= y_{2}= z_{2}=y_{3} - y_{5}=z_{3} - z_{5}=y_{4} - y_{6} \,, \\
	0&=z_{4} - z_{6}=x_{5}= z_{5}=x_{2} - x_{3}=z_{2} - z_{3}=x_{3} - x_{4}=y_{3} - y_{4}=x_{5} - x_{6}\,, \\
	0&=y_{5} - y_{6}=x_{7} - x_{6} - y_{7} + y_{6}=x_{7} - x_{6} - z_{7} + z_{6}=  x_{4}  -y_{4}= x_{4} - z_{4}\,,
\end{align*}
with the general solution parametrized by $t\in\RR$
\[ (p_1,\dots,p_7) = ((0,0,0), (t, 0, 0), (t, t, 0), (t, t, t), (0, t, 0), (0, t, t), (-t, 0, 0)) . \]
The solution is injective for $t\neq 0$.
If we take $t=1$ and the parametrization of the deltoid from Example~\ref{ex:deltoid},
then we obtain an algebraic motion of the graph such that its projection to the 4-cycle induced by $\{1,2,3,4\}$
is precisely the motion of the deltoid.
Any other parametrization of the 4-cycle also yields an algebraic motion of the whole graph.
Figure~\ref{fig:Q1pfls} illustrates using the deltoid and also a general quadrilateral.
We remark that the triangle $\{1,2,7\}$ is degenerated independently of the choice of parametrization of the 4-cycle.
Moreover, the 4-cycles $\{1,2,3,5\}$, $\{3,5,6,4\}$ and $\{1,4,6,7\}$ are always parallelograms.
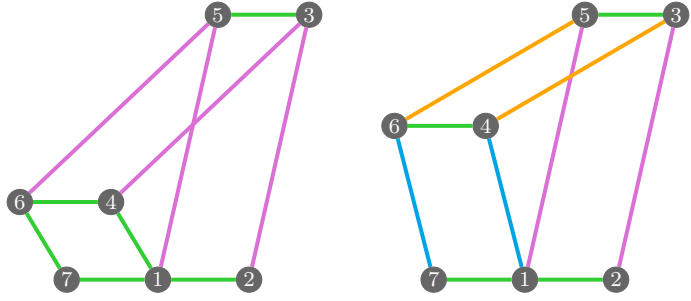
\begin{figure}[htb]
	\centering	
		\begin{tikzpicture}[scale=1.2]
			\node[lnode] (0) at (-1, 0) {7};
			\node[lnode] (1) at (0, 0) {1};
			\node[lnode] (2) at (1, 0) {2};
			\node[lnode] (3) at (68/41, 120/41) {3};
			\node[lnode] (4) at (-611/1189, 1020/1189) {4};
			\node[lnode] (5) at (27/41, 120/41) {5};
			\node[lnode] (6) at (-1800/1189, 1020/1189) {6};
		    \draw[edge, col1] (0)edge(1) (1)edge(2) (4)edge(6) (3)edge(5) (0)edge(6) (1)edge(4);
			\draw[edge, col2]  (2)edge(3) (3)edge(4)  (1)edge(5) (5)edge(6);
		\end{tikzpicture}
		\hspace{0.5cm}
		\begin{tikzpicture}[scale=1.2]
			\node[lnode] (0) at (-1, 0) {7};
			\node[lnode] (1) at (0, 0) {1};
			\node[lnode] (2) at (1, 0) {2};
			\node[lnode] (3) at (68/41, 120/41) {3};
			\node[lnode] (4) at (-511/1189, 2020/1189) {4};
			\node[lnode] (5) at (27/41, 120/41) {5};
			\node[lnode] (6) at (-1700/1189, 2020/1189) {6};
		    \draw[edge, col1] (0)edge(1) (1)edge(2) (4)edge(6) (3)edge(5) ;
			\draw[edge, col2]  (2)edge(3)  (1)edge(5);
			\draw[edge, col3] (3)edge(4) (5)edge(6);
			\draw[edge, col4] (0)edge(6) (1)edge(4);
		\end{tikzpicture}
	\caption{Embeddings of the graph $Q_1$ compatible with proper flexible labelings 
	induced by a deltoid and a general quadrilateral
	(colors indicate edges with same lengths).
	Note that the vertices 1,2,7 form a degenerate triangle.}
\label{fig:Q1pfls}
\end{figure}
\end{exmp}

By applying the described procedure to all pairs of NAC-colorings for the graphs in the list,
we obtain the following:
\begin{cor}
	\label{cor:twoNacGraphs}
	The graphs $Q_1, \dots, Q_6$, see Figure~\ref{fig:constDistClosures} or~\ref{fig:graphsTwoNAC}, are movable.
\end{cor}
\begin{proof}
	Figure~\ref{fig:graphsTwoNAC} shows the pairs of NAC-colorings of the graphs in $Q_1, \dots, Q_6$
	that can be used to construct an injective embedding in $\RR^3$ analogously to Example~\ref{ex:Q1}.
	Hence, they are movable by Lemma~\ref{lem:twoNacs}.
\end{proof}

\begin{figure}[htb]
	\centering
	\begin{tabular}{cccccc}
	\begin{tikzpicture}[scale=1*0.8]
		    \node[vertex] (1) at (0.500, 0.866) {};
		    \node[vertex] (0) at (-0.500, 0.866) {};
		    \node[vertex] (2) at (-1.00, 0.000) {};
		    \node[vertex] (3) at (1.00, 0.000) {};
		    \node[vertex] (5) at (0.500, -0.866) {};
		    \node[vertex] (6) at (-0.500, -0.866) {};
		    \node[vertex] (7) at (0.000, 0.000) {};
		    \draw[redge] (2)edge(6) (1)edge(7)  ;
		    \draw[bedge] (0)edge(7) (3)edge(5)  ;
		    \draw[bedge] (5)edge(7) (5)edge(6) (6)edge(7)  (1)edge(2) (0)edge(3)   ;
		    \draw[redge] (1)edge(3) (0)edge(2)  ;
		\end{tikzpicture}		
		&
		\begin{tikzpicture}[scale=1*0.8, yscale=-1]
		    \node[vertex] (0) at (-1.00, 0.000) {};
		    \node[vertex] (1) at (1.00, 0.000) {};
		    \node[vertex] (2) at (0.500, 0.866) {};
		    \node[vertex] (3) at (0.500, -0.866) {};
		    \node[vertex] (4) at (-0.500, -0.866) {};
		    \node[vertex] (5) at (-0.500, 0.866) {};
		    \node[vertex] (6) at (0.000, -0.300) {};
		    \node[vertex] (7) at (0.000, 0.300) {};
		    \draw[redge] (1)edge(3) (6)edge(7) (0)edge(5)  ;
		    \draw[redge] (3)edge(4) (3)edge(7) (4)edge(7) (1)edge(6)  ;
		    \draw[bedge] (0)edge(6) (5)edge(7) (2)edge(5) (2)edge(7)  ;
		    \draw[bedge] (1)edge(2) (0)edge(4)  ;
		\end{tikzpicture}
		&
		\begin{tikzpicture}[scale=1*0.8]
		    \node[vertex] (1) at (0.500, 0.866) {};
		    \node[vertex] (0) at (-0.500, 0.866) {};
		    \node[vertex] (2) at (-1.00, 0.000) {};
		    \node[vertex] (3) at (1.00, 0.000) {};
		    \node[vertex] (4) at (0.000, -0.433) {};
		    \node[vertex] (5) at (0.500, -0.866) {};
		    \node[vertex] (6) at (-0.500, -0.866) {};
		    \node[vertex] (7) at (0.000, 0.000) {};
		    \draw[redge] (2)edge(6) (1)edge(7)  ;
		    \draw[bedge] (0)edge(7) (3)edge(5)  ;
		    \draw[bedge] (5)edge(7) (5)edge(6) (6)edge(7) (4)edge(6) (4)edge(7) (1)edge(2) (0)edge(3) (4)edge(5)  ;
		    \draw[redge] (1)edge(3) (0)edge(2)  ;
		\end{tikzpicture}		
		&
		\begin{tikzpicture}[scale=1*0.8]
		    \node[vertex] (0) at (0.500, -0.866) {};
		    \node[vertex] (1) at (-0.500, -0.866) {};
		    \node[vertex] (2) at (0.500, 0.866) {};
		    \node[vertex] (3) at (-0.500, 0.866) {};
		    \node[vertex] (4) at (-1.00, 0.000) {};
		    \node[vertex] (5) at (1.00, 0.000) {};
		    \node[vertex] (6) at (0.300, -0.100) {};
		    \node[vertex] (7) at (-0.300, -0.100) {};
		    \draw[redge] (0)edge(7) (1)edge(4)  ;
		    \draw[bedge] (0)edge(5) (1)edge(6)  ;
		    \draw[bedge] (3)edge(4) (3)edge(7) (4)edge(7) (5)edge(6) (2)edge(6) (2)edge(5) (0)edge(1)  ;
		    \draw[redge] (2)edge(4) (6)edge(7) (3)edge(5)  ;
		\end{tikzpicture}
		&
		\begin{tikzpicture}[scale=1*0.8]
		    \node[vertex] (0) at (-0.588, -0.809) {};
		    \node[vertex] (1) at (0.588, -0.809) {};
		    \node[vertex] (2) at (0.951, 0.309) {};
		    \node[vertex] (3) at (-0.951, 0.309) {};
		    \node[vertex] (4) at (0.235, 0.324) {};
		    \node[vertex] (5) at (-0.235, 0.324) {};
		    \node[vertex] (6) at (0.000, -0.400) {};
		    \node[vertex] (7) at (0.000, 1.00) {};
		    \draw[bedge] (1)edge(2) (0)edge(3)  ;
		    \draw[bedge] (0)edge(6) (1)edge(6) (0)edge(1)  ;
		    \draw[redge] (2)edge(4) (5)edge(6) (4)edge(7) (2)edge(7)  ;
		    \draw[redge] (5)edge(7) (3)edge(5) (3)edge(7) (4)edge(6)  ;
		\end{tikzpicture}
		&
		\begin{tikzpicture}[scale=1*0.8]
		    \node[vertex] (0) at (0.000, -0.433) {};
		    \node[vertex] (1) at (-0.500, -0.866) {};
		    \node[vertex] (2) at (0.500, -0.866) {};
		    \node[vertex] (3) at (1.00, 0.000) {};
		    \node[vertex] (4) at (-1.00, 0.000) {};
		    \node[vertex] (5) at (0.000, 0.000) {};
		    \node[vertex] (6) at (0.500, 0.866) {};
		    \node[vertex] (7) at (-0.500, 0.866) {};
		    \draw[redge] (3)edge(7) (4)edge(6)  ;
		    \draw[bedge] (0)edge(5) (1)edge(7) (4)edge(7) (3)edge(6) (2)edge(3) (2)edge(6) (1)edge(4)  ;
		    \draw[redge] (5)edge(6) (0)edge(2)  ;
		    \draw[bedge] (5)edge(7) (0)edge(1)  ;
		\end{tikzpicture}
		\\
		\begin{tikzpicture}[scale=1*0.8]
		    \node[vertex] (1) at (0.500, 0.866) {};
		    \node[vertex] (0) at (-0.500, 0.866) {};
		    \node[vertex] (2) at (-1.00, 0.000) {};
		    \node[vertex] (3) at (1.00, 0.000) {};
		    \node[vertex] (5) at (0.500, -0.866) {};
		    \node[vertex] (6) at (-0.500, -0.866) {};
		    \node[vertex] (7) at (0.000, 0.000) {};
		    \draw[bedge] (2)edge(6) (1)edge(7)  ;
		    \draw[redge] (0)edge(7) (3)edge(5)  ;
		    \draw[bedge] (5)edge(7) (5)edge(6) (6)edge(7)  (1)edge(2) (0)edge(3)   ;
		    \draw[redge] (1)edge(3) (0)edge(2)  ;
		\end{tikzpicture}		
		&
		\begin{tikzpicture}[scale=1*0.8, yscale=-1]
		    \node[vertex] (0) at (-1.00, 0.000) {};
		    \node[vertex] (1) at (1.00, 0.000) {};
		    \node[vertex] (2) at (0.500, 0.866) {};
		    \node[vertex] (3) at (0.500, -0.866) {};
		    \node[vertex] (4) at (-0.500, -0.866) {};
		    \node[vertex] (5) at (-0.500, 0.866) {};
		    \node[vertex] (6) at (0.000, -0.300) {};
		    \node[vertex] (7) at (0.000, 0.300) {};
		    \draw[bedge] (1)edge(3) (6)edge(7) (0)edge(5)  ;
		    \draw[redge] (3)edge(4) (3)edge(7) (4)edge(7) (1)edge(6)  ;
		    \draw[bedge] (0)edge(6) (5)edge(7) (2)edge(5) (2)edge(7)  ;
		    \draw[redge] (1)edge(2) (0)edge(4)  ;
		\end{tikzpicture}
		&
		\begin{tikzpicture}[scale=1*0.8]
		    \node[vertex] (1) at (0.500, 0.866) {};
		    \node[vertex] (0) at (-0.500, 0.866) {};
		    \node[vertex] (2) at (-1.00, 0.000) {};
		    \node[vertex] (3) at (1.00, 0.000) {};
		    \node[vertex] (4) at (0.000, -0.433) {};
		    \node[vertex] (5) at (0.500, -0.866) {};
		    \node[vertex] (6) at (-0.500, -0.866) {};
		    \node[vertex] (7) at (0.000, 0.000) {};
		    \draw[bedge] (2)edge(6) (1)edge(7)  ;
		    \draw[redge] (0)edge(7) (3)edge(5)  ;
		    \draw[bedge] (5)edge(7) (5)edge(6) (6)edge(7) (4)edge(6) (4)edge(7) (1)edge(2) (0)edge(3) (4)edge(5)  ;
		    \draw[redge] (1)edge(3) (0)edge(2)  ;
		\end{tikzpicture}		
		&
		\begin{tikzpicture}[scale=1*0.8]
		    \node[vertex] (0) at (0.500, -0.866) {};
		    \node[vertex] (1) at (-0.500, -0.866) {};
		    \node[vertex] (2) at (0.500, 0.866) {};
		    \node[vertex] (3) at (-0.500, 0.866) {};
		    \node[vertex] (4) at (-1.00, 0.000) {};
		    \node[vertex] (5) at (1.00, 0.000) {};
		    \node[vertex] (6) at (0.300, -0.100) {};
		    \node[vertex] (7) at (-0.300, -0.100) {};
		    \draw[redge] (0)edge(7) (1)edge(4)  ;
		    \draw[redge] (0)edge(5) (1)edge(6)  ;
		    \draw[bedge] (3)edge(4) (3)edge(7) (4)edge(7) (5)edge(6) (2)edge(6) (2)edge(5) (0)edge(1)  ;
		    \draw[bedge] (2)edge(4) (6)edge(7) (3)edge(5)  ;
		\end{tikzpicture}
		&
		\begin{tikzpicture}[scale=1*0.8]
		    \node[vertex] (0) at (-0.588, -0.809) {};
		    \node[vertex] (1) at (0.588, -0.809) {};
		    \node[vertex] (2) at (0.951, 0.309) {};
		    \node[vertex] (3) at (-0.951, 0.309) {};
		    \node[vertex] (4) at (0.235, 0.324) {};
		    \node[vertex] (5) at (-0.235, 0.324) {};
		    \node[vertex] (6) at (0.000, -0.400) {};
		    \node[vertex] (7) at (0.000, 1.00) {};
		    \draw[redge] (1)edge(2) (0)edge(3)  ;
		    \draw[bedge] (0)edge(6) (1)edge(6) (0)edge(1)  ;
		    \draw[bedge] (2)edge(4) (5)edge(6) (4)edge(7) (2)edge(7)  ;
		    \draw[redge] (5)edge(7) (3)edge(5) (3)edge(7) (4)edge(6)  ;
		\end{tikzpicture}
		&
		\begin{tikzpicture}[scale=1*0.8]
		    \node[vertex] (0) at (0.000, -0.433) {};
		    \node[vertex] (1) at (-0.500, -0.866) {};
		    \node[vertex] (2) at (0.500, -0.866) {};
		    \node[vertex] (3) at (1.00, 0.000) {};
		    \node[vertex] (4) at (-1.00, 0.000) {};
		    \node[vertex] (5) at (0.000, 0.000) {};
		    \node[vertex] (6) at (0.500, 0.866) {};
		    \node[vertex] (7) at (-0.500, 0.866) {};
		    \draw[redge] (3)edge(7) (4)edge(6)  ;
		    \draw[bedge] (0)edge(5) (1)edge(7) (4)edge(7) (3)edge(6) (2)edge(3) (2)edge(6) (1)edge(4)  ;
		    \draw[bedge] (5)edge(6) (0)edge(2)  ;
		    \draw[redge] (5)edge(7) (0)edge(1)  ;
		\end{tikzpicture}
		\\
		$Q_1$ & $Q_2$ &$Q_3$ &$Q_4$ &$Q_5$ &$Q_6$ 
	\end{tabular}
	\caption{Pairs of NAC-colorings used for construction of injective embeddings in $\RR^3$ 
	satisfying the assumption of Lemma~\ref{lem:twoNacs}.}
	\label{fig:graphsTwoNAC}
\end{figure}
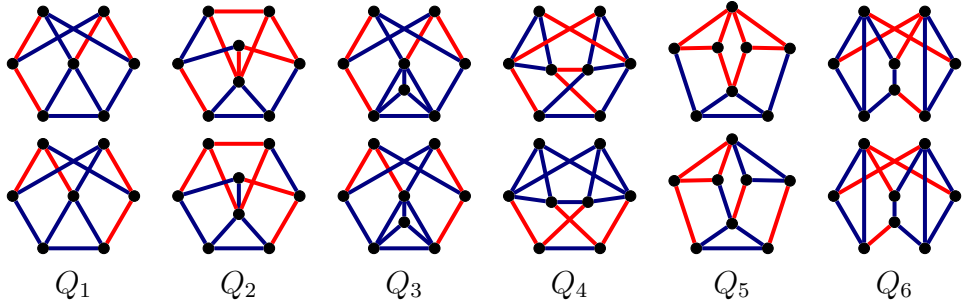

For the graphs $S_1,\dots,S_4$, we take advantage of the fact 
that they contain other graphs in the list as subgraphs.
The next lemma formalizes the general construction based on movable subgraphs.
\begin{lem}\label{lem:combiningTwoGraphs}
	Let $G$ be a graph.	
	Let $G_1$ and $G_2$ be two subgraphs of $G$ such that $V_G=V_{G_1}\cup V_{G_2}, E_G=E_{G_1}\cup E_{G_2}$
	and $E_{G_1}\cap E_{G_2}\neq\emptyset$.
	Let $W=V_{G_1}\cap V_{G_2}$.
	Let $\lambda_1$ and $\lambda_2$ be proper flexible labelings of $G_1$ and $G_2$ respectively.
	If there are algebraic motions $\C_1$ of~$(G_1,\lambda_1)$ and $\C_2$ of $(G_2,\lambda_2)$
	such that:
	\begin{enumerate}
		\item\label{it:combiningTwoGraphs:proj} the projections of $\C_1$ and $\C_2$ to $W$ are the same, and
		\item for all $v_1\in V_{G_1}\setminus W$ and $v_2\in V_{G_2}\setminus W$, the projections of $\C_1$ to $v_1$ and $\C_2$ to $v_2$ are different,
	\end{enumerate}	 
	then there exists a proper flexible labeling of $G$.
\end{lem}
\begin{proof}
	We define a labeling $\lambda$ of $G$ by $\lambda|_{E_{G_1}}=\lambda_1$ and $\lambda|_{E_{G_2}}=\lambda_2$.
	This is well-defined, since $\lambda_1|_{E_{G_1}\cap E_{G_2}}=\lambda_2|_{E_{G_1}\cap E_{G_2}}$ by \ref{it:combiningTwoGraphs:proj}.
	Now, every realization in the projection of $\C_1$ to $W$ can be extended to a realization of $G$ that is compatible with $\lambda$.
	Hence, $\lambda$ is flexible. It is also proper, since all extended realizations are injective by the second assumption.
\end{proof}

Now, we identify the suitable subgraphs and motions for $S_1, S_2$ and $S_3$.
Movability of~$S_4$ does not follow from the previous lemma, but it is straightforward.
\begin{cor}
	\label{cor:movableByCombination}
	The graphs $S_1, S_2,S_3$ and $S_4$, see Figure~\ref{fig:constDistClosures} or~\ref{fig:S1S2S3}, are movable.
\end{cor}
\begin{proof}
	Figure~\ref{fig:S1S2S3} shows vertex-labelings of the graphs $S_1,S_2$ and $S_3$ that are used in the proof.
	The labelings given by the displayed edge lengths are actually proper flexible.
	Edges with same lengths have the same color.
		
	Notice that the subgraphs $G_1$ and $G_2$ of $S_1$ induced by the vertices $\{1,\dots,6\}$
	and $\{3,\dots,8\}$ are isomorphic to $L_1$ and $K_{3,3}$ respectively.
	Since $G_1$ and $G_2$ satisfy the assumptions of Lemma~\ref{cor:movableByCombination},
	it is sufficient to take proper flexible labelings of $G_1$ and $G_2$, 
	given by Lemma~\ref{lem:oneNac} and~\ref{lem:bipartiteGraphs},
	such that the quadrilateral (3,4,5,6) in both graphs moves as a non-degenerated rhombus, 
	i.e., $\lambda(3,4)=\lambda(4,5)=\lambda(5,6)=\lambda(3,6)$.  

	Recall that a proper flexible labeling of $K_{4,4}$ according to Dixon~II is induced by placing the nodes 
	of the two partition sets to the vertices of two cocentric rectangles in orthogonal position.
	By removing two vertices, one can easily obtain a motion of $K_{3,3}$.
	
	The graph $S_2$ has a subgraph $H_{Q_1}$ induced by vertices $\{1,\dots,7\}$, which is isomorphic to $Q_1$,
	and $H_{K_{3,3}}$ induced by $\{1,\dots,5,8\}$ isomorphic to $K_{3,3}$.
	We consider a proper flexible labeling of the subgraph $H_{K_{3,3}}$
	with an algebraic motion by Dixon II according to Figure~\ref{fig:S1S2S3}.
	Now, we can use the motion of the 4-cycle $\{1,2,3,4\}$ to construct a motion of $H_{Q_1}$
	following Example~\ref{ex:Q1}. 
	Since the 4-cycle $\{1,2,3,5\}$ is a parallelogram in the motions of $H_{K_{3,3}}$ and $H_{Q_1}$,
	the subgraphs satisfy the assumption of Lemma~\ref{lem:combiningTwoGraphs}. Hence,~$S_2$ is movable.
	
	Similarly, we construct a proper flexible labeling of $S_3$, since the vertices $\{1,\dots,7\}$ and 
	$\{1,3,4,5,6,8\}$ induce subgraphs isomorphic to $Q_1$ and $K_{3,3}$, respectively.
	See Figure~\ref{fig:S1S2S3} for placing the vertices according to Dixon~II.
	Now, the 4-cycle $\{1,4,3,5\}$ is used to construct the motion according to Example~\ref{ex:Q1}.

	A proper flexible labeling of $S_4$ can be clearly obtained by extending a proper flexible labeling 
	of its $K_{3,3}$ subgraph.
\end{proof}
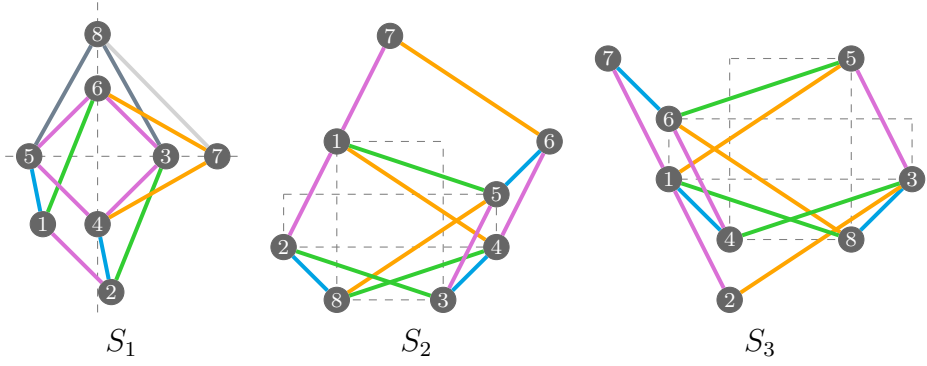
\begin{figure}[htb]
	\centering	
	\begin{tabular}{ccc}
		\begin{tikzpicture}[scale=0.9, rotate=90]
			\draw[gridl, dashed] (-2.3,0)edge(2.3,0);
			\draw[gridl, dashed] (0,1.35)edge(0,-2.05);
			\node[lnode] (2) at (1.8, 0) {8};
			\node[lnode] (5) at (-1., 0) {4};
			\node[lnode] (7) at (1, 0) {6};
			\node[lnode] (1) at (0, -1.75) {7};
			\node[lnode] (6) at (0,  1) {5};
			\node[lnode] (4) at (0, -1) {3};
			\node[lnode] (0) at (-1,  0.8) {1};
			\node[lnode] (3) at (-2, -0.2) {2};
			\draw[edge, col1] (7)edge(0) (3)edge(4);
			\draw[edge, col2]  (6)edge(5) (5)edge(4) (7)edge(4) (7)edge(6) (0)edge(3);
			\draw[edge, col5] (2)edge(4) (2)edge(6);
			\draw[edge, col3] (1)edge(5) (7)edge(1) ;
			\draw[edge, col6] (2)edge(1);
			\draw[edge, col4] (0)edge(6)  (5)edge(3) ;
		\end{tikzpicture}
		&
		\begin{tikzpicture}[scale=0.7]
			\draw[gridl, dashed] (1,0)edge(3,0) (1,3)edge(3,3) (0,1)edge(4,1) (0,2)edge(4,2);
			\draw[gridl, dashed] (0,1)edge(0,2) (4,1)edge(4,2) (1,0)edge(1,3) (3,0)edge(3,3); 
			\node[lnode] (7) at (1,3) {1};
			\node[lnode] (5) at (4,1) {4};
			\node[lnode] (3) at (3,0) {3};
			\node[lnode] (6) at (0,1) {2};
			\node[lnode] (2) at (1,0) {8};
			\node[lnode] (4) at (4,2) {5};
			\node[lnode] (1) at (2,5) {7};
			\node[lnode] (0) at (5,3) {6};
			\draw[edge, col4] (5)edge(3) (2)edge(6) (4)edge(0);
			\draw[edge, col3]  (2)edge(4) (0)edge(1) (5)edge(7);
			\draw[edge, col1] (2)edge(5) (7)edge(4) (3)edge(6);
			\draw[edge, col2] (7)edge(6) (3)edge(4) (1)edge(7) (5)edge(0);
		\end{tikzpicture}
		&
		\begin{tikzpicture}[scale=0.8, xscale=-1, yscale=1]
			\draw[gridl, dashed] (1,0)edge(3,0) (1,3)edge(3,3) (0,1)edge(4,1) (0,2)edge(4,2);
			\draw[gridl, dashed] (0,1)edge(0,2) (4,1)edge(4,2) (1,0)edge(1,3) (3,0)edge(3,3); 
			\node[lnode] (4) at (1,3) {5};
			\node[lnode] (7) at (4,1) {1};
			\node[lnode] (3) at (3,0) {4};
			\node[lnode] (6) at (0,1) {3};
			\node[lnode] (2) at (1,0) {8};
			\node[lnode] (5) at (4,2) {6};
			\node[lnode] (0) at (3,-1) {2};
			\node[lnode] (1) at (5,3) {7};
			\draw[edge, col4] (7)edge(3) (2)edge(6) (5)edge(1);
			\draw[edge, col3]  (2)edge(5) (7)edge(4) (0)edge(6);
			\draw[edge, col1] (2)edge(7) (4)edge(5) (3)edge(6);
			\draw[edge, col2] (4)edge(6) (3)edge(5) (7)edge(1) (0)edge(7);
		\end{tikzpicture}
		\\
		$S_1$ &$S_2$ &$S_3$
	\end{tabular}		
	\caption{The graphs $S_1,S_2$ and $S_3$ with proper flexible labelings (same colors mean same lengths).
		 Note that the vertices 1,2,7 in $S_2$ and $S_3$ form a degenerate triangle.}
	\label{fig:S1S2S3}
\end{figure}

Finally, only the graph $S_5$ is missing to be proven to be movable.
Unfortunately, none of the previous constructions applies in this case.
Hence, we provide a parametrization of its algebraic motion ad hoc.
\begin{lem}
	\label{lem:S5}
	The graph $S_5$ is movable.
\end{lem}
\begin{proof}
In order to construct a proper flexible labeling for the graph $S_5$, 
we assume the following: the triangles $(1,2,3)$ and $(1,4,5)$ are degenerated into lines,
the quadrilaterals $(1,4,6,2)$ and $(1,4,7,3)$ are antiparallelograms,
the quadrilateral  $(4,7,8,6)$ is a rhombus and 
the quadrilaterals $(4,5,8,7)$ and $(4,5,8,6)$ are deltoids, see Figure~\ref{fig:S5}.
We scale the lengths so that $\lambda_{1,4}=1$  and $\lambda_{1,2}=:a>1$.
Now, we define an injective realization $\rho_\theta$ parametrized by the position of vertex $4$.
Let
\begin{equation*}
	\rho_\theta(1) = (0,0), \quad \rho_\theta(2) = (-a,0), \quad \rho_\theta(3) = (a,0), \quad \rho_\theta(4) = (\cos\theta,\sin\theta)\,.
\end{equation*}
Since the coordinates of a missing vertex of an antiparallelogram 
can be obtained by folding the parallelogram with the same edges along a diagonal, we get
\begin{align*}
\rho_\theta(6) &=\left(-\frac{a^{3} + {\left(a^{2} - 1\right)} \cos\theta - a}{a^{2} + 2 a \cos\theta + 1}, 
		\frac{{\left(1 - a^{2}\right)} \sin\theta}{a^{2} + 2  a \cos\theta + 1}\right)\,, \\
\rho_\theta(7) &=\left(\phantom{-}\frac{a^{3} - {\left(a^{2} - 1\right)} \cos\theta - a}{a^{2} - 2 a \cos\theta + 1}, 
		 \frac{{\left(1 - a^{2}\right)} \sin\theta}{a^{2} - 2  a \cos\theta + 1}\right)\,.
\end{align*}
The intersection of the line given by $\rho_\theta(6)$ and $\rho_\theta(7)$ with the line given by $\rho_\theta(1)$ and~$\rho_\theta(4)$ gives
\begin{equation*}
\rho_\theta(5) = \frac{1 - a^{2}}{a^{2} + 1}(\cos\theta,\sin\theta)\,.
\end{equation*}
The position of $8$ can be easily obtained by the fact that $(4,7,8,6)$ is a rhombus:
\begin{equation*}
\rho_\theta(8) = \left( \frac{\left(\left(a^{2} -  1\right)^2 - 4 a^{2}\left(\sin\theta\right)^{2}\right) \cos\left(\theta\right)}
			{\left(a^{2} -  1\right)^2 + 4 a^{2} \left(\sin\theta\right)^{2}},
 \frac{{-\left(3  a^{4} - 4  a^{2} \left(\cos\theta\right)^{2} + 2  a^{2} - 1\right)} \sin\theta}
 			{\left(a^{2} -  1\right)^2 + 4 a^{2} \left(\sin\theta\right)^{2}}
 		\right)\,.
\end{equation*}
One can verify that the induced labeling $\lambda$ is independent of $\theta$ and hence it is flexible:
\begin{align*}
	\lambda_{1,4}&=\lambda_{2,6}=\lambda_{3,7}=1\,, \qquad \lambda_{2,3}=2a\,, \\
	\lambda_{1,2}&=\lambda_{1,3}=\lambda_{4,6}=\lambda_{4,7}=\lambda_{6,8}=\lambda_{7,8}=a \,, \\
	\lambda_{1,5}&=\frac{a^2-1}{a^2+1} \,, \qquad
	\lambda_{4,5}=\lambda_{5,8}=\frac{2a^2}{a^2+1} = \lambda_{1,5}+\lambda_{1,4}\,.
\end{align*}
The labeling is proper for a generic $a$.
\end{proof}
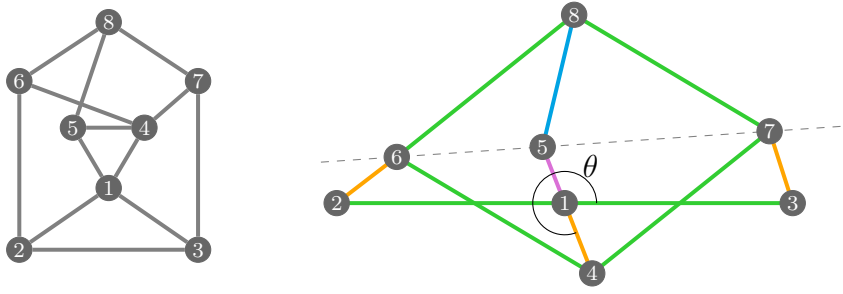
\begin{figure}[htb]
	\centering			
	\begin{tikzpicture}[scale=2]
		    \node[lnode] (0) at (0.000, 0.700) {8};
		    \node[lnode] (1) at (-0.588, -0.809) {2};
		    \node[lnode] (2) at (0.588, -0.809) {3};
		    \node[lnode] (3) at (0.588, 0.309) {7};
		    \node[lnode] (4) at (-0.588, 0.309) {6};
		    \node[lnode] (5) at (-0.235, 0) {5};
		    \node[lnode] (6) at (0.000, -0.400) {1};
		    \node[lnode] (7) at (0.235, 0) {4};
		    \draw[edge] (0)edge(3) (0)edge(4) (0)edge(5) (1)edge(2) (1)edge(4) (1)edge(6) (2)edge(3) (2)edge(6) (3)edge(7) (4)edge(7) (5)edge(6) (5)edge(7) (6)edge(7)  ;
	\begin{scope}[xshift=3cm,yshift=-0.5cm, scale=1.5]
		\coordinate (8) at (0.04149, 0.8324);
		\coordinate (3) at (1, 0.0);
		\coordinate (2) at (-1, 0.0);
		\coordinate (6) at (-0.7365, 0.2041);
		\coordinate (7) at (0.8987, 0.3175);
		\coordinate (5) at (-0.0966, 0.2485);
		\coordinate (1) at (0.0, 0.0);
		\coordinate (4) at (0.1207, -0.3106);
		\draw[gridl, dashed, shorten >= -1.cm, shorten <=-1.cm] (6)edge(7);
		\draw[edge, col1] (8)edge(6) (8)edge(7)  (3)edge(1)  (2)edge(1) (6)edge(4) (4)edge(7);
		\draw[edge, col3] (2)edge(6) (3)edge(7) (4)edge(1) ;
		\draw[edge, col2]   (5)edge(1);
		\draw[edge, col4]  (8)edge(5) ;
		\node[lnode] at (8) {8};
		\node[lnode] at (3) {3};
		\node[lnode] at (2) {2};
		\node[lnode] at (6) {6};
		\node[lnode] at (7) {7};
		\node[lnode] at (5) {5};
		\node[lnode] at (1) {1};
		\node[lnode] at (4) {4};
		\coordinate (9) at (0.3, 0);
		\tkzLabelAngle[color=black,pos=0.2](5,1,3){$\theta$};
		\tkzMarkAngle[color=black,size=0.14](3,1,4);
				\end{scope}
	\end{tikzpicture}
	\caption{The graph $S_5$ and its embedding inducing a proper flexible labeling. The same colors mean same edge lengths.}
	\label{fig:S5}
\end{figure}

Since all graphs in the list were proven to be movable, 
we can conclude that the necessary condition is also sufficient up to 8 vertices.
\begin{cor}
	Let $G$ be a graph with at most $8$ vertices.
	The graph~$G$ is movable if and only if the constant distance closure $\cdc{G}$ is not complete.
\end{cor}
\begin{proof}
	We can assume that $G$ is spanned by a Laman graph,
	otherwise there exists a generic proper flexible labeling.
	By Lemma~\ref{lem:degTwo}, we can assume that $G$ has no vertex of degree two.
	Corollary~\ref{cor:necesarryCondProper} gives the necessary condition for movability. 
	For the opposite implication, 
	Theorem~\ref{thm:listEightVertices} lists 
	all constant distance closures that are not complete,
	Lemma~\ref{lem:bipartiteGraphs} and \ref{lem:S5}, and Corollary~\ref{cor:oneNacGraphs},
	 \ref{cor:twoNacGraphs} and \ref{cor:movableByCombination}
	show that all these graphs are movable. Hence, also all their subgraphs are movable.
\end{proof}

Based on the previous corollary, one might want to conjecture that
the statement holds independently of the number of vertices.
Nevertheless, the graph $G_{25}$ in Figure~\ref{fig:25vertexGraph} serves as a counter example.

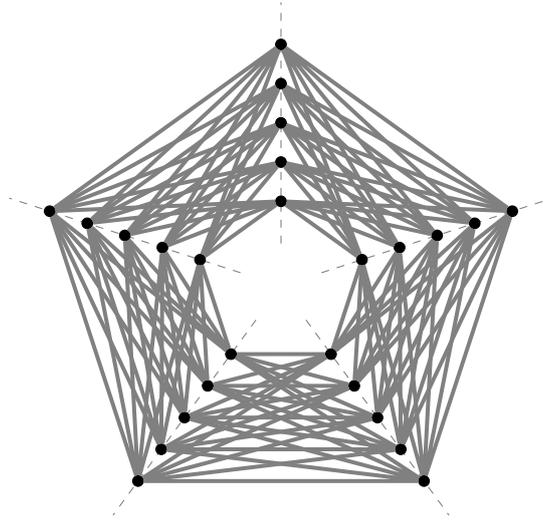
\begin{figure}[htb]
  \begin{center}
    \begin{tikzpicture}[scale=0.8]
    
    	\pgfmathsetmacro\cosA{cos(90)};
		\pgfmathsetmacro\sinA{sin(90)};
		\pgfmathsetmacro\cosB{cos(162)};
		\pgfmathsetmacro\sinB{sin(162)};
		\pgfmathsetmacro\cosC{cos(234)};
		\pgfmathsetmacro\sinC{sin(234)};
		\pgfmathsetmacro\cosD{cos(306)};
		\pgfmathsetmacro\sinD{sin(306)};
		\pgfmathsetmacro\cosE{cos(18)};
		\pgfmathsetmacro\sinE{sin(18)};
		\pgfmathsetmacro\rA{1.4};
		\pgfmathsetmacro\rB{2.05};
		\pgfmathsetmacro\rC{2.7};
		\pgfmathsetmacro\rD{3.35};
		\pgfmathsetmacro\rE{4.0};
		\pgfmathsetmacro\rF{4.7};
		\pgfmathsetmacro\rZ{0.7};
		
		\draw[gridl, dashed] (\rZ*\cosA, \rZ*\sinA)to(\rF*\cosA, \rF*\sinA);
		\draw[gridl, dashed] (\rZ*\cosB, \rZ*\sinB)to(\rF*\cosB, \rF*\sinB);
		\draw[gridl, dashed] (\rZ*\cosC, \rZ*\sinC)to(\rF*\cosC, \rF*\sinC);
		\draw[gridl, dashed] (\rZ*\cosD, \rZ*\sinD)to(\rF*\cosD, \rF*\sinD);
		\draw[gridl, dashed] (\rZ*\cosE, \rZ*\sinE)to(\rF*\cosE, \rF*\sinE);
		
		\node[vertex] (A1) at (\rA*\cosA, \rA*\sinA) {};
		\node[vertex] (A2) at (\rB*\cosA, \rB*\sinA) {};
		\node[vertex] (A3) at (\rC*\cosA, \rC*\sinA) {};
		\node[vertex] (A4) at (\rD*\cosA, \rD*\sinA) {};
		\node[vertex] (A5) at (\rE*\cosA, \rE*\sinA) {};

		\node[vertex] (B1) at (\rA*\cosB, \rA*\sinB) {};
		\node[vertex] (B2) at (\rB*\cosB, \rB*\sinB) {};
		\node[vertex] (B3) at (\rC*\cosB, \rC*\sinB) {};
		\node[vertex] (B4) at (\rD*\cosB, \rD*\sinB) {};
		\node[vertex] (B5) at (\rE*\cosB, \rE*\sinB) {};
		
		\node[vertex] (C1) at (\rA*\cosC, \rA*\sinC) {};
		\node[vertex] (C2) at (\rB*\cosC, \rB*\sinC) {};
		\node[vertex] (C3) at (\rC*\cosC, \rC*\sinC) {};
		\node[vertex] (C4) at (\rD*\cosC, \rD*\sinC) {};
		\node[vertex] (C5) at (\rE*\cosC, \rE*\sinC) {};
		
		\node[vertex] (D1) at (\rA*\cosD, \rA*\sinD) {};
		\node[vertex] (D2) at (\rB*\cosD, \rB*\sinD) {};
		\node[vertex] (D3) at (\rC*\cosD, \rC*\sinD) {};
		\node[vertex] (D4) at (\rD*\cosD, \rD*\sinD) {};
		\node[vertex] (D5) at (\rE*\cosD, \rE*\sinD) {};
		
		\node[vertex] (E1) at (\rA*\cosE, \rA*\sinE) {};
		\node[vertex] (E2) at (\rB*\cosE, \rB*\sinE) {};
		\node[vertex] (E3) at (\rC*\cosE, \rC*\sinE) {};
		\node[vertex] (E4) at (\rD*\cosE, \rD*\sinE) {};
		\node[vertex] (E5) at (\rE*\cosE, \rE*\sinE) {};
		
		\draw[edge] (A1)to(B1) (A1)to(B2) (A1)to(B3) (A1)to(B4) (A1)to(B5);
		\draw[edge] (A2)to(B1) (A2)to(B2) (A2)to(B3) (A2)to(B4) (A2)to(B5);
		\draw[edge] (A3)to(B1) (A3)to(B2) (A3)to(B3) (A3)to(B4) (A3)to(B5);
		\draw[edge] (A4)to(B1) (A4)to(B2) (A4)to(B3) (A4)to(B4) (A4)to(B5);
		\draw[edge] (A5)to(B1) (A5)to(B2) (A5)to(B3) (A5)to(B4) (A5)to(B5);
				
		\draw[edge] (A1)to(E1) (A1)to(E2) (A1)to(E3) (A1)to(E4) (A1)to(E5);
		\draw[edge] (A2)to(E1) (A2)to(E2) (A2)to(E3) (A2)to(E4) (A2)to(E5);
		\draw[edge] (A3)to(E1) (A3)to(E2) (A3)to(E3) (A3)to(E4) (A3)to(E5);
		\draw[edge] (A4)to(E1) (A4)to(E2) (A4)to(E3) (A4)to(E4) (A4)to(E5);
		\draw[edge] (A5)to(E1) (A5)to(E2) (A5)to(E3) (A5)to(E4) (A5)to(E5);
				
		\draw[edge] (C1)to(B1) (C1)to(B2) (C1)to(B3) (C1)to(B4) (C1)to(B5);
		\draw[edge] (C2)to(B1) (C2)to(B2) (C2)to(B3) (C2)to(B4) (C2)to(B5);
		\draw[edge] (C3)to(B1) (C3)to(B2) (C3)to(B3) (C3)to(B4) (C3)to(B5);
		\draw[edge] (C4)to(B1) (C4)to(B2) (C4)to(B3) (C4)to(B4) (C4)to(B5);
		\draw[edge] (C5)to(B1) (C5)to(B2) (C5)to(B3) (C5)to(B4) (C5)to(B5);
				
		\draw[edge] (D1)to(E1) (D1)to(E2) (D1)to(E3) (D1)to(E4) (D1)to(E5);
		\draw[edge] (D2)to(E1) (D2)to(E2) (D2)to(E3) (D2)to(E4) (D2)to(E5);
		\draw[edge] (D3)to(E1) (D3)to(E2) (D3)to(E3) (D3)to(E4) (D3)to(E5);
		\draw[edge] (D4)to(E1) (D4)to(E2) (D4)to(E3) (D4)to(E4) (D4)to(E5);
		\draw[edge] (D5)to(E1) (D5)to(E2) (D5)to(E3) (D5)to(E4) (D5)to(E5);
				
		\draw[edge] (D1)to(C1) (D1)to(C2) (D1)to(C3) (D1)to(C4) (D1)to(C5);
		\draw[edge] (D2)to(C1) (D2)to(C2) (D2)to(C3) (D2)to(C4) (D2)to(C5);
		\draw[edge] (D3)to(C1) (D3)to(C2) (D3)to(C3) (D3)to(C4) (D3)to(C5);
		\draw[edge] (D4)to(C1) (D4)to(C2) (D4)to(C3) (D4)to(C4) (D4)to(C5);
		\draw[edge] (D5)to(C1) (D5)to(C2) (D5)to(C3) (D5)to(C4) (D5)to(C5);
    \end{tikzpicture}

  \caption{Graph $G_{25}$}
  \label{fig:25vertexGraph}
  \end{center}
\end{figure}

This graph was proposed by Tibor Jord\'an as a counter example for some conjectures 
characterizing movable graphs within informal discussions with his students.
The constant distance closure of $G_{25}$ is the graph itself,
since there is no monochromatic path of length at least two: for every two incident edges $uv$ and $vw$, 
there exists a NAC-coloring $\delta$ such that $\delta(uv)\neq\delta(vw)$.
Namely, we can define $\delta$ by $\delta(e)=\blue{}$ if and only if $w$ is a vertex of $e$.
Hence, the necessary condition is satisfied.
An explanation that $G_{25}$ is not movable is the following:
it contains five subgraphs isomorphic to the bipartite graph~$K_{5,5}$,
each of them induced by the vertices on two neighboring lines in the figure.
The only way to construct a proper flexible labeling of $K_{5,5}$, with partition sets $V_1$ and $V_2$,
is placing the vertices of $V_1$ on a line and the vertices of $V_2$ on another line that is perpendicular to the first one~\cite{Maehara2001}.
Therefore, constructing a proper flexible labeling of $G_{25}$
would require that the vertices on every two neighboring lines in the figure are on perpendicular lines,
which is not possible.

\section*{Conclusion}
The newly introduced notion of active NAC-colorings bridges the combinatorial properties of a movable graph 
with its motion.
Motivated by invariance of the motion under adding new edges whose endpoints are connected 
by a path that is monochromatic in all active NAC-colorings,
the constant distance closure of a graph is defined purely combinatorially.
This augmented graph being non-complete serves as a necessary condition of movability of the original graph.
We focused on the graphs up to 8 vertices satisfying the condition and developed tools showing that all are movable.

Since it was shown that the necessary condition is not always sufficient,
the question on a (combinatorial) characterization of movability remains open.
Similarly, the characterization of all possible algebraic motions of a given graph is subject to future research.
A particularly interesting open problem is to determine possible subsets of active NAC-colorings
and then construct a corresponding proper flexible labeling.

\section*{Acknowledgment}
We thank Meera Sitharam for the discussion which led to Theorem~\ref{thm:treeDecomposable} and happened during the workshop on
Rigidity and Flexibility of Geometric Structures organized 
by the Erwin Sch\"odinger International Institute for Mathematics and Physics in Vienna in September~2018.
Furthermore, we thank Tibor Jord\'an for discussions on the counterexample.

% \bibliographystyle{plain}
% \bibliography{nrl}

\end{document}